\documentclass[12pt]{article}
\pdfoutput=1
\usepackage{amsmath,amssymb,amsthm,bm,graphicx,float,array,multirow,multicol,rotfloat,caption,subcaption,hyperref,cleveref,enumerate,geometry,mathdots,adjustbox,booktabs,parskip,mathtools,tikz,tikz-cd,pdflscape,csquotes,lscape,rotating,empheq,mathrsfs,longtable,braket}
\usepackage[para]{threeparttable}
\usepackage[all]{xy}
\usepackage[normalem]{ulem}
\usepackage[numbers,sort&compress]{natbib}
\numberwithin{equation}{section}
\numberwithin{figure}{section}
\allowdisplaybreaks
\makeatletter
\usetikzlibrary{arrows,shapes.misc,positioning,decorations.pathmorphing,decorations.markings,matrix,patterns,patterns.meta,calc}

\tikzset{rectangle/.style={draw,circle,inner sep=1pt},
big arrow/.style={decoration={markings,mark=at position 1 with {\arrow[scale=1.5,#1]{>}}},postaction={decorate},shorten >=0.4pt},
scale cd/.style={every label/.append style={scale=#1}, cells={nodes={scale=#1}}}}

\DeclareRobustCommand{\loongrightarrow}{\DOTSB\relbar\joinrel\relbar\joinrel\rightarrow}
\theoremstyle{plain}
\newtheorem*{thm*}{Theorem}
\newtheorem{thm}{Theorem}[section]

\newtheorem{lem}[thm]{Lemma}
\newtheorem{prop}[thm]{Proposition}

\theoremstyle{definition}

\newtheorem*{defn*}{Definition}

\newtheorem{rem}[thm]{Remark}

\newtheorem*{t:main}{Theorem \ref{t:main}}
\newtheorem*{t:approx}{Theorem \ref{t:approx}}
\newcommand{\mc}[1]{\mathcal{#1}}
\newcommand{\A}{\mathcal{A}}
\newcommand{\B}{\mathcal{B}}
\newcommand{\cH}{\mathcal{H}}
\newcommand{\K}{\mathcal{K}}

\newcommand{\BH}{\mc{B}(\mathcal{H})}
\newcommand{\BK}{\mc{B}(\mathcal{K})}

\newcommand{\C}{\mathbb{C}}
\newcommand{\R}{\mathbb{R}}
\newcommand{\N}{\mathbb{N}}

\newcommand{\la}{\langle}
\newcommand{\ra}{\rangle}

\newcommand{\ten}{\otimes}

\newcommand{\ep}{\varepsilon}

\newcommand{\id}{\textnormal{id}}
\newcommand{\om}{\omega}
\newcommand{\Om}{\Omega}
\newcommand{\vphi}{\varphi}
\providecommand{\norm}[1]{\lVert#1\rVert}

\makeatother

\begin{document}

\begin{titlepage}
 \vspace*{-3cm} 
 \begin{flushright}
 \end{flushright}
\begin{center}
\vspace{2cm}
{\LARGE\bfseries Algebraic approach to spacetime bulk reconstruction}
\vspace{1.2cm}

{\large Jason Crann$^{1}$ and Monica Jinwoo Kang$^{2}$\\}
\vspace{.7cm}
{$^1$ School of Mathematics \& Statistics, Carleton University}\\
{Ottawa, ON H1S 5B6, Canada}\par
\vspace{.2cm}
{$^2$ Department of Physics and Astronomy, University of Pennsylvania}\\
{Philadelphia, PA 19104, U.S.A.}\par
\vspace{.5cm}

\scalebox{1}{\tt jasoncrann@cunet.carleton.ca, monica6@sas.upenn.edu}\par
\vspace{2cm}
\textbf{Abstract}
\end{center}

Motivated by the theory of holographic quantum error correction in the anti-de Sitter/conformal field theory (AdS/CFT) correspondence, together with the kink transform conjecture on the bulk AdS description of boundary cocycle flow, we characterize (approximate) complementary recovery in terms of (approximate) intertwining of bulk and boundary cocycle derivatives. Using the geometric modular structure in vacuum AdS, we establish an operator algebraic subregion-subregion duality of boundary causal diamonds and bulk causal wedges for Klein-Gordon fields in the universal cover of AdS. Our results suggest that, from an algebraic perspective, the kink transform is bulk cocycle flow, which (in the above case) induces the bulk geometry via geometric modular action and the corresponding notion of time.
As a by-product, we find that if the von Neumann algebra of a boundary CFT subregion is a type $\mathrm{III}_1$ factor with an ergodic vacuum, then the von Neumann algebra of the corresponding dual bulk subregion, is either $\mathbb{C}1$ (with a one-dimensional Hilbert space) or a type $\mathrm{III}_1$ factor.

\vfill 
\end{titlepage}

\tableofcontents
\newpage

\section{Introduction}\label{sec:intro}


Over the past decade, perspectives from quantum information have shed valuable insight into the conjectured anti-de Sitter space/conformal field theory (AdS/CFT) correspondence \cite{Maldacena}. Its quantum field theoretic formulation (see, e.g., \cite{Witten,Hamilton:2006az}), in the semi-classical limit, postulates a duality between a quantum field theory coupled to gravity in (the universal cover of) $(d+1)$-dimensional anti-de Sitter space, $AdS_{d+1}$, and a conformal field theory on its $d$-dimensional conformal boundary $\partial AdS_{d+1}$. This duality leads to a correspondence between bulk geometry (coupled to fields) and boundary quantum information (in CFT states/observables). 

A primary example is the Ryu--Takayanagi formula \cite{RT06}, which connects the entanglement entropy of a state over a boundary region $A\subseteq\partial AdS_{d+1}$ to the area of a certain minimal surface $\gamma$ in the bulk anchored at $A$ (see Figure \ref{fig:bulkreconstruction} below). Another example is subregion-subregion duality, in which local bulk observables are encoded as quantum error-correcting codes on a subspace of the boundary CFT Hilbert space and bulk AdS locality emerges from complementary recovery of the encoding \cite{Dong:2016eik,Almheiri:2014lwa,Harlow:2016vwg}. This novel quantum error correction perspective has been steadily developed in recent years, and continues to shed insight into aspects of AdS/CFT, including the semi-classical limit of quantum gravity and the emergence of spacetime structure (see, e.g., \cite{Gesteau:2020wrk,Harlow:2018fse,Almheiri:2019psf,Penington:2019npb,Almheiri:2019hni}).

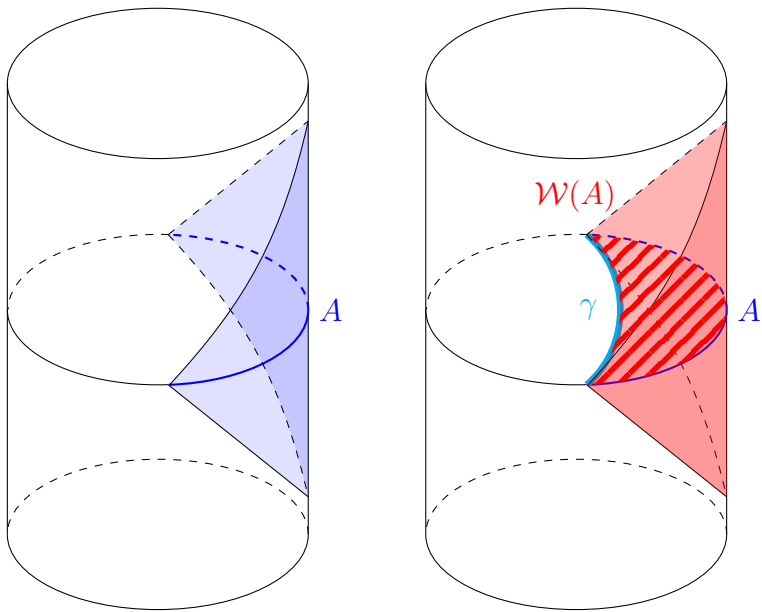
\begin{figure}[H]
    \centering
    \vspace{5pt}
    \begin{tikzpicture}
        \draw (-2,3)--(-2,-3);
        \draw (2,3)--(2,-3);
        \draw (0,3) ellipse (2cm and 1cm);
        \draw[dashed] (2,-3) arc (0:180:2cm and 1cm);
        \draw (-2,-3) arc (180:360:2cm and 1cm);
        \draw[dashed] (.15,.99) arc (86:180:2cm and 1cm);
        \draw (-2,-.01) arc (180:274:2cm and 1cm);
        \draw[blue,thick] (.14,-1.01) arc (-86:0:2cm and 1cm);
        \draw[blue,thick,dashed] (2,-.01) arc (0:86:2cm and 1cm);
        \coordinate (P1) at ($(0,-0.01)+(86:2cm and 1cm)$);
        \coordinate (P2) at ($(0,-0.01)+(-86:2cm and 1cm)$);
        \coordinate (P22) at ($(0.01,-0.02)+(-86:2cm and 1cm)$);
        \coordinate[label={[blue]0:$A$}] (A) at (2,-0.01);
        \draw[dashed] (P1) to (2,2.5);
        \draw (P22) to[bend right=15] (2,2.5);
        \draw[dashed] (P1) to[bend left=15] (2,-2.5);
        \draw (P22) to (2,-2.5);
        \path[draw=none, fill=blue, opacity=0.12] (P22) to[bend right=15] (2,2.5) to (2,-2.5) to (P22);
        \path[draw=none, fill=blue, opacity=0.12] (P1) to[bend left=15] (2,-2.5) to (2,2.5) to (P1);
    \end{tikzpicture}
    \qquad
    \begin{tikzpicture}
        \draw (-2,3)--(-2,-3);
        \draw (2,3)--(2,-3);
        \draw (0,3) ellipse (2cm and 1cm);
        \draw[dashed] (2,-3) arc (0:180:2cm and 1cm);
        \draw (-2,-3) arc (180:360:2cm and 1cm);
        \draw[dashed] (.15,.99) arc (86:180:2cm and 1cm);
        \draw (-2,-.01) arc (180:274:2cm and 1cm);
        \draw[blue,thick] (.14,-1.01) arc (-86:0:2cm and 1cm);
        \draw[blue,thick,dashed] (2,-.01) arc (0:86:2cm and 1cm);
        \coordinate (P1) at ($(0,-0.01)+(86:2cm and 1cm)$);
        \coordinate (P2) at ($(0,-0.01)+(-86:2cm and 1cm)$);
        \coordinate (P22) at ($(0.01,-0.02)+(-86:2cm and 1cm)$);
        \coordinate[label={[blue]0:$A$}] (A) at (2,-0.01);
        \coordinate[label={[cyan]0:$\gamma$}] (G) at (-.1,-0.01);
        \draw[cyan,line width=2.5pt] (P1) to[bend left=50] (P2);
        \draw[dashed] (P1) to (2,2.5);
        \draw (P22) to[bend right=15] (2,2.5);
        \draw[dashed] (P1) to[bend left=15] (2,-2.5);
        \draw (P22) to (2,-2.5);
        \path[draw=none, fill=red, opacity=0.4] (P22) to[bend right=15] (2,2.5) to (2,-2.5) to (P22);
        \path[draw=none, fill=red, opacity=0.3] (P1) to[bend left=50] (P2) to[bend right=15] (2,2.5) to (P1);
        \draw[draw=none, pattern={Lines[angle=45,distance={8pt/sqrt(2)},line width=2pt]},pattern color=red] (P1) to[bend left=50] (P2) to (.14,-1.01) arc (-86:86:2cm and 1cm);
        \coordinate[label={[red]0:$\mathcal{W}(A)$}] (W) at (-.7,1.5);
    \end{tikzpicture}
    \caption{Geometric depiction of the Ryu--Takayanagi surface $\gamma$ and the reconstructable bulk wedge $\mathcal{W}(A)$ of a boundary region $A$ under subregion-subregion duality.}
    \label{fig:bulkreconstruction}
\end{figure}

However, the vast majority of (rigorous) examples witnessing this error correcting structure are finite-dimensional, often realized via tensor networks (e.g., \cite{Happy,JE}). Although such toy models have revealed various holographic features expected of quantum gravity, they are nevertheless discrete and can be difficult to analyze in connection with the continuous nature of quantum field theory.\footnote{There are tensor network models that addresses its settings in infinite-dimensional Hilbert spaces with its attempt to capture some continuous behavior. For examples, see \cite{Kang:2019dfi,Gesteau:2020rtg,Gesteau:2020hoz,SpacetimeTN,SpacetimeTN2,Gesteau:2021jzp}.}

Given the infinite-dimensional/operator-algebraic nature of quantum fields, it is natural to study the quantum error correcting nature of AdS/CFT directly by combining the frameworks of operator algebra quantum error correction \cite{BKK1,BKK2} and algebraic quantum field theory \cite{Haag}. This approach was taken in \cite{KK}, which generalized the results of \cite{Harlow:2016vwg} to arbitrary von Neumann algebras. Their work has since been expanded and developed in \cite{Gesteau:2020rtg,Faulkner:2020hzi,Gesteau:2021jzp}. In these works, the equivalence between (exact) bulk reconstruction (i.e., complementary recovery) and the equality of bulk and boundary relative entropies are shown using operator-algebraic techniques. 
Note also the recent preprint which establishes a Ryu-Takayangi type formula for quantum systems modelled on factors of type $\mathrm{I}$ or $\mathrm{II}$ \cite{XZ}.

In this work, we build on this paradigm by giving an operator-algebraic description of the spacetime aspect of bulk reconstruction. Specifically, if $V:\cH\hookrightarrow\K$ is an isometry between the bulk and boundary Hilbert spaces, and $\B$ and $\A$ are bulk and boundary von Neumann algebras on $\K$ and $\cH$, respectively, we show that complementary recovery of $\B$ is equivalent to $V$ intertwining all cocycle derivatives of vector states on $\B$. Viewing cocycle conjugation as a local extension of modular flow (valid only on a fixed state, see Remark \ref{r:rem}), we arrive at a (state-dependent) localized time reconstruction compatible with modular time and subregion-subregion duality. The precise statement is Theorem \ref{t:main}, copied below (notation defined and explained in detail in Section \ref{s:modular}):

\begin{thm}\label{t:main} 
Let $V:\cH\hookrightarrow\K$ be an isometry between Hilbert spaces $\cH$ and $\K$. Let $\A$ and $\B$ be von Neumann algebras on $\K$ and $\cH$, respectively. Suppose there is a cyclic and separating state vector $\Omega\in\cH$ for $\B$ such that $V\Omega\in\K$ is cyclic and separating for $\A$. Then the following conditions are equivalent:
\begin{enumerate}
    \item (Complementary Recovery) $\B$ and $\B'$ are correctable subalgebras for the complementary channels $V^*(\cdot)V:\A\to\BH$ and $V^*(\cdot)V:\A'\to\BH$, respectively.
    \item (Preservation of Connes Cocycle Flow) For all states $\psi\in\cH$, we have $$V[D\om_{\Om}:D\om_\psi]_t=[D\om_{V\Om}:D\om_{V\psi}]_tV, \ \  V[D\om_{\Om}':D\om_{\psi}']_t=[D\om_{V\Om}':D\om_{V\psi}']_tV, \ \ \ t\in\R.$$
    \item (Preservation of Relative Entropy) For all states $\psi\in\cH$, 
    $$S_{\A}(\om_{V\psi},\om_{V\Om})=S_{\B}(\om_{\psi},\om_{\Om}) \ \ \ S_{\A'}(\om_{V\psi}',\om_{V\Om}')=S_{\B'}(\om_{\psi}',\om_{\vphi}').$$
\end{enumerate}
\end{thm}
The above result also suggests that the kink transform of \cite{Bousso}, which they conjecture to be the bulk geometric description of boundary cocycle flow, is precisely bulk cocycle flow, at least in exact recovery settings. As we show in Section \ref{s:4}, exact complementary recovery is satisfied for
$$\mc{A}:=\pi_{\om}(\mc{A}_{bd}(\mc{D}))''\quad\text{and}\quad\mc{B}:=\pi_{\om}(\mc{A}_{bulk}(\mc{W}))'',$$
where $\mc{A}_{bd}(\mc{D})$ and $\mc{A}_{bulk}(\mc{W})$ are Weyl $C^*$-algebras associated to Klein-Gordon fields over boundary causal diamonds $\mc{D}$ and causal wedges $\mc{W}$ in the universal cover of AdS, where $\omega$ is the global AdS vacuum (see Theorem \ref{t:KG} and the surrounding discussion for a detailed statement). This result not only proves operator algebraic subregion-subregion duality for such regions, the geometric implementations of their bulk and boundary modular flows \cite{BEM,BGL} match the geometric nature of the kink transform, providing additional evidence for the conjecture of \cite{Bousso}. 

To consider the higher-order effect in $G_N$, one has to incorporate non-pertubative gravity corrections that require approximate recovery \cite{Kelly:2016edc}, in comparison to the exact recovery of Theorem \ref{t:main}. The approximate recovery in infinite-dimensional Hilbert spaces has been discussed in \cite{Gesteau:2021jzp} through an operator-algebraic approach, where the reconstruction errors are controlled by a function of $N$, which vanishes in the large $N$ limit. A weaker version of an approximate recovery was recently considered in \cite{FL}, using asymptotically isometric encodings. In a similar spirit, our final result is an approximate version of Theorem \ref{t:main} (see $\S3.3$ for the relation with \cite{FL}):

\begin{thm}\label{t:approx} 
Let $(V_n:\cH\hookrightarrow\K_n)_{n\in\N}$ be a sequence of isometries between Hilbert spaces $\cH$ and $\K_n$. Let $\A_n$ and $\B$ be von Neumann algebras on $\K_n$ and $\cH$, respectively. Suppose there is a cyclic and separating state vector $\Omega\in\cH$ for $\B$ such that $V_n\Omega\in\K_n$ is cyclic and separating for each $\A_n$. Consider the following conditions:
\begin{enumerate}
\item There exist sequences $(R_n:\B\to\A_n)_{n\in \N}, (R_n':\B'\to\A_n')_{n\in \N}$ of channels such that
$$wot-\lim_nV_n^*R_n(b)V_n=b, \ \ wot-\lim V_n^*R'_n(b')V_n=b'$$
for all $b\in\B$ and $b'\in\B'$.
\item For all cyclic and separating states $\psi\in\cH$ (for $\B$), we have 
\begin{align*}
sot-\lim_nV_n[D\om_{\Om}:D\om_\psi]_t&-[D\om_{V_n\Om}:D\om_{V_n\psi}]_tV_n=0\\
sot-\lim_nV_n[D\om_{\Om}':D\om_{\psi}']_t&-[D\om_{V_n\Om}':D\om_{V_n\psi}']_tV_n=0, \ \ \ t\in\R.
\end{align*}
\item For all bounded sequences $(a_n)\in\prod_{n\in\N}\A_n$, $(a'_n)\in\prod_{n\in\N}\A'_n$, and  and $b\in\B$, $b'\in\B$,
$$wot-\lim_n[V_n^*a_nV_n,b']=0, \ \ \ wot-\lim_n[V_n^*a'_nV_n,b]=0.$$
\end{enumerate}
Then $(1)\Rightarrow(2)\Rightarrow(3)$, and if $\B$ is hyperfinite, the three conditions are equivalent. 
\end{thm}

The organization of the paper is as follows. In section \ref{s:preliminary}, we review preliminaries on cocycle derivatives, modular flow, relative entropy, as well as operator algebra quantum error correction. We establish the equivalence of exact complementary recovery with cocycle intertwining in Section \ref{s:exact} and discuss the emergent type $\mathrm{III}_1$ von Neumann algebras in Section \ref{s:III1}. We formulate and prove approximate complementary recovery in Section \ref{s:approx}. Utilizing the framework of \cite{DW}, in Section \ref{s:4} we prove subregion-subregion duality between boundary causal diamonds and bulk causal wedges for Klein--Gordon fields in the universal cover of AdS, providing a physically relevant, mathematically rigorous instance of (infinite-dimensional) operator-algebraic complementary recovery. We conclude with an outlook on future directions in Section \ref{s:outlook}.

\section{Preliminaries}\label{s:preliminary}
In this section, we review the necessary preliminaries from von Neumann algebras and operator-algebraic quantum error correction that will be applied in the context of holography.

\subsection{Relative modular theory} \label{s:modular}

We begin with an overview of Connes' cocycle derivatives and relative entropy in von Neumann algebras, following the presentation and notation of \cite[\S4,\S5]{OP}. All Hilbert spaces in this paper are assumed separable. The space of normal linear functionals on a von Neumann algebra $\A$, i.e., the predual of $\A$, is denoted $\A_*$.

Let $\A$ be a von Neumann algebra on a Hilbert space $\cH$, and let $\omega'$ be a positive normal functional in $(\A')_*$. A vector $\xi\in\cH$ is \textit{$\omega'$-bounded} if the map 
\begin{align}
    R^{\omega'}(\xi):\cH_{\omega'}\ni\Lambda_{\omega'}(a')\mapsto a'\xi\in\cH, \quad a'\in\A',
\end{align}
is bounded, where $\Lambda_{\omega'}:\A'\to \cH_{\omega'}$ is the Gelfand--Naimark--Segal (GNS) map of $\omega'$. We let $\mc{D}(\mathcal{H},\omega')$ denote the linear space of $\omega'$-bounded vectors. The closure of $\mc{D}(\cH,\omega')$ is $s(\omega')\cH$, where $s(\omega')$, the support projection of $\omega'\in (\A')_*$, is the smallest projection $p\in\A'$ such that $\omega'(p)=\omega'(1)$. Given $\xi,\eta\in\mc{D}(\cH,\omega')$, it follows that $R^{\omega'}(\xi)R^{\omega'}(\eta)^*\in\A$. 

For a positive element $\vphi\in\A_*$, the function
\begin{align}
    q(\xi+\eta)=\vphi(R^{\omega'}(\xi)R^{\omega'}(\eta)^*)
\end{align}
for $\xi\in\mc{D}(\mathcal{H},\omega')$ and $\eta\in\mc{D}(\mathcal{H},\omega')^{\perp}$ is a densely defined lower semi-continuous quadratic form on $\cH$. It is therefore closable and by the form representation theorem \cite[Theorem VIII.15]{RS72} (see also \cite[Theorem 5]{Connes}) there is a positive self-adjoint operator $\Delta(\vphi/\omega')$ such that 
\begin{enumerate}[\hspace{1em}(1)]
\item $\norm{\Delta(\vphi/\omega')^{1/2}\xi}^2=q(\xi)\,,\quad\xi\in\mc{D}(\cH,\omega')$,
\item $\mc{D}(\cH,\omega')$ is a core for $\Delta(\vphi/\omega')^{1/2}$. 
\end{enumerate}
The operator $\Delta(\vphi/\omega')$ is called the \textit{spatial derivative} of $\vphi$ with respect to $\omega'$. It satisfies the following properties:
\begin{enumerate}[\hspace{1em}(1)]
\item
$\mathrm{supp}\,\Delta(\vphi/\omega')=s(\vphi)s(\om')$.
\item $\Delta((\vphi_1+\vphi_2)/\om')=\Delta(\vphi_1/\om')+\Delta(\vphi_2/\om')$, where the sum on the right hand side is the form sum of positive operators. If $s(\vphi_1)\perp s(\vphi_2)$, we have an orthogonal sum.
\item $\Delta(\vphi/\om')^z=\Delta(\om'/\vphi)^{-z}$, $z\in\C$, where we use the convention that for a positive self-adjoint operator $A$, $A^z$ is the sum of 0 on $\mathrm{supp}(A)^{\perp}$ and the usual power $A^z$ on $\mathrm{supp}(A)$.
\end{enumerate}

Spatial derivatives were introduced by Connes \cite{Connes} as a generalization of Araki's relative modular operator \cite{Araki,ArakiII}. Indeed, if $(\A,\cH,J,\mc{P})$ is a standard form \cite{Haagerup75}, and $\vphi$ and $\omega$ are normal states on $\A$, then there exist unique vectors $\xi_\vphi,\xi_\om\in\mc{P}$ such that
\begin{align}
    \vphi=\omega_{\xi_{\vphi}}|_{\A}\,,\quad \omega=\omega_{\xi_{\om}}|_{\A}\,.
\end{align}
On the domain $\A\xi_{\om}+\A\xi_{\omega}^{\perp}$, the relative Tomita operator
\begin{align}
    S^0_{\vphi,\omega}(a\xi_\om+\eta) = s(\om)a^*\xi_{\vphi}\,, \quad a\in\A\,,\quad \eta\in\A\xi_{\omega}^{\perp}\,,
\end{align}
is closable. Its closure $S_{\vphi,\omega}$ admits a polar decomposition 
\begin{align}
    S_{\vphi,\omega} = J_{\vphi,\omega}\Delta(\vphi,\omega)^{1/2}\,.
\end{align}
The \textit{relative modular operator} $\Delta(\vphi,\omega)$ coincides with the spatial derivative $\Delta(\vphi/\omega'_{\xi_\om})$ of $\vphi$ with respect to the induced vector state $\omega'_{\xi_\om}$ on $\A'$.

Returning to the general case of a von Neumann algebra $\A$ on $\cH$, if $\vphi$ and $\omega$ are normal states on $\A$, $\om_0$ is a faithful normal state on $\A'$, and $z\in\C$, then the operator $\Delta(\vphi/\om_0')^z\Delta(\om/\om_0')^{-z}$ is independent of $\om_0'$ \cite{Araki73}. Moreover, the bounded operator 
\begin{align}
    [D\vphi:D\om]_{t}:= \Delta(\vphi/\om_0')^{it}\Delta(\om/\om_0')^{-it}\in\A, \quad t\in\R .
\end{align}
This one-parameter family is known as the \emph{Connes (or Radon-Nikodym) cocycle} of $\vphi$ and $\om$. We often write 
\begin{align}
    u_t=[D\vphi:D\om]_{t} ,
\end{align}
when the context is clear. 
\begin{thm}[Takesaki {\cite[Theorem VIII.3.3]{T2}}] \label{thm:T2}
If both $\vphi$ and $\om$ are faithful, each $u_t$ is a unitary in $\A$, and the family $(u_t)_{t\in\R}$ satisfies the following properties
\begin{enumerate}
\item $(u_t)_{t\in\R}$ is $\sigma$-strongly continuous,
\item $u_{s+t}=u_s\sigma_s^\om(u_t)\,,\quad s,t\in\R$,
\item $\sigma^\vphi_t(a)=u_t\sigma^\om_t(a)u_t^*\,,\quad t\in\R$,
\item for each $a,b\in\A$, there is a bounded continuous function $F$ on the closed horizontal strip $\overline{\mathbb{S}}$ bounded by $\R$ and $\R+i$ which is holomorphic on the open strip $\mathbb{S}$ such that 
$$F(t)=\vphi(u_t\sigma_t^{\om}(y)x), \ \ \ F(t+i)=\om(xu_t\sigma^{\om}(y)), \ \ \ t\in\R,$$
\end{enumerate}
where $\sigma^\vphi_t$ and $\sigma^\om_t$ are the modular automorphism groups of the vector states $\vphi$ and $\om$ on $\A$. Moreover, the family $(u_t)_{t\in\R}$ is uniquely determined by the KMS condition (4).
\end{thm}

\begin{rem}\label{r:rem} Suppose, in addition to the hypothesis of Theorem \ref{thm:T2}, that $\om=\om_\psi$ is a vector state. Letting $\psi_t:=u_t\psi$ for $t\in\R$, the expectation values of observables $a'\in\A'$ satisfy 
$$\la\psi_t,a'\psi_t\ra=\la\psi,a'\psi\ra,$$ since each $u_t$ is in $\A$. From the property (3) and the fact that $\om_\psi=\om_\psi\circ\sigma^{\om}_t$ (on $\A$), for observables $a\in\A$, we have
\begin{equation*}
\la\psi_t,a\psi_t\ra=\la\psi,u_t^*au_t\psi\ra=\la\psi,u_t^*\sigma^{\vphi}_t(\sigma^{\vphi}_{-t}(a))u_t\psi\ra=\la\psi,\sigma^{\om}_t(\sigma^{\vphi}_{-t}(a))\psi\ra=\la\psi,\sigma^{\vphi}_{-t}(a)\psi\ra.
\end{equation*}
Thus, 
\begin{equation}\label{e:local}\om_\psi\circ\mathrm{Ad}([D\vphi:D\om]_t^*)|_{\mc{A}}=\om_\psi\circ\sigma^\vphi_{-t}, \ \ \ t\in\R.\end{equation}
In other words, the evolution of the $\psi$-expectation values of $\A$ under the $\vphi-\om_{\psi}$-cocycle flow coincide with those under the (inverse) modular flow of $\vphi$. This observation was made in \cite{Bousso}, where it reflects the geometric nature of the kink transform (introduced in that paper) at the level of operator algebras. From a quantum channel perspective, equation \eqref{e:local} means that $\mathrm{Ad}(u_t^*)$ defines an $\mc{A}$-local (i.e., $\mc{A}'$-bimodule map leaving $\mc{A}$ globally invariant) $\om_\psi$-extension of $(\sigma^\vphi_{-t})_{t\in\R}$ from $\mc{A}$ to all of $\BH$. 
\end{rem}

The relative entropy is defined through spatial derivatives as follows. Let $\A$ be a von Neumann algebra on $\cH$, let $\omega$ and $\vphi$ be normal states on $\A$, and suppose $\omega=\om_\psi|_{\A}$, $\psi\in\cH$. The
\textit{relative entropy} $S(\omega,\vphi)$ of $\omega$ and $\vphi$ is 
\begin{align}
    S(\omega,\vphi)= \begin{cases}
        +\infty & \text{if }\psi\notin s(\vphi)\cH\\
        -\la\log(\Delta(\vphi/\om'_\psi))\psi,\psi\ra & \text{otherwise}
    \end{cases}
\end{align}
It is known that $S(\omega,\vphi)$ is independent of the representing vector $\psi$ for $\omega$.  From the monotonic limit $\lim_{t\to0^+}t^{-1}(\lambda^t-1)=\log\lambda$, the monotone convergence theorem implies that
\begin{equation}\label{e:log}S(\omega,\vphi)=-\lim_{t\to0^+}t^{-1}(\norm{\Delta(\vphi/\om'_{\psi})^{t/2}\psi}^2-1).\end{equation}

\subsection{Operator-algebraic quantum error correction}

Given von Neumann algebras $\A\subseteq\BK$, and $\B\subseteq\BH$, a \textit{quantum channel} from $\A$ to $\B$ is a normal, unital completely positive map $\mc{E}:\A\to\B$.  A von Neumann subalgebra $\mc{C}\subseteq\B$ is \textit{private} for $\mc{E}$ if $\mc{E}(\A)\subseteq\mc{C}'$ \cite{CKLT}. The subalgebra $\mc{C}$ is \textit{correctable} for $\mc{E}$ if there exists a quantum channel $\mc{R}:\mc{C}\to\A$ such that $\mc{E}\circ\mc{R}=\id_{\mc{C}}$ \cite{BKK1,BKK2}. When $\mc{C}=\B$ and 
\begin{equation}\label{e:sufficient}
    \om_i\circ\mc{E}\circ\mc{R}=\om_i
\end{equation}
for all $\om_i$ in a family $(\om_i)_{i\in I}$ of normal states in $\B_*$, $\mc{E}$ is said to be \textit{sufficient} for the family $(\om_i)_{i\in I}$ \cite{OP,Petz}. If, in addition to satisfying equation \eqref{e:sufficient}, $\B$ is a subalgebra of $\A$ and $\mc{R}$ is the inclusion map, the subalgebra $\B$ is said to be \textit{sufficient} for the family $(\om_i)_{i\in I}$ (via the quantum channel $\mc{E}$).

Given a quantum channel $\mc{E}:\A\to\B$, by Stinespring's representation theorem \cite{Stinespring}, there exists a Hilbert space $\mc{L}$, a normal unital $*$-homomorphism $\pi:\A\rightarrow\mc{B}(\mc{L})$ and an isometry $V:\mc{H}\to\mc{L}$ such that 
\begin{align}
    \mc{E}(a) = V^*\pi(a)V, \ \ \ a\in\A .
\end{align}
The triple $(\pi,V,\mc{L})$ forms a \emph{Stinespring representation} of $\mc{E}$, which is unique up to
a conjugation by a partial isometry
in the following sense: if $(\pi_1,V_1,\mc{L}_1)$ and $(\pi_2,V_2,\mc{L}_2)$ are
Stinespring representations for $\mc{E}$, then there is a partial isometry
$U:\mc{L}_1\rightarrow \mc{L}_2$ such that
\begin{equation}
    UV_1=V_2, \quad U^*V_2=V_1, \quad U\pi_1(a)=\pi_2(a)U
\end{equation}
for all $a\in\A$. If $(\pi_1,V_1,\mc{L}_1)$ yields a \emph{minimal}
Stinespring representation, meaning that
$\pi_1(\A)V_1\K$ is a dense subspace of $\mc{L}_1$, the map $U$ above is necessarily an isometry, and any two minimal Stinespring representations for $\mc{E}$ are
unitarily equivalent. When $\A$ and $\B$ are properly infinite, minimal Stinespring representations exist with $\mc{L}=\mc{H}$, $\pi(\A)\subseteq\B$, and $V\in\B$ \cite[Theorem 2.10]{Longo18}.

A \textit{complementary channel} of a quantum channel $\mc{E}:\A\to\B$ is defined as 
$$\mc{E}^c:\pi(\A)'\ni x\mapsto V^*xV\in\BH,$$
where $(\pi,V,\mc{L})$ is a Stinespring representation for $\mc{E}$. By above, complements are unique up to conjugation by a partial isometry. The (exact) complementarity theorem below, which relies on Arveson's commutant lifiting \cite[Theorem 1.3.1]{Arv69}, can be seen as an operator algebraic generalization of the Knill-Laflamme error correcting conditions \cite{KL97} (see \cite{BKK1,BKK2,CKLT} for an extended discussion).

\begin{thm}\label{t:comp} (\cite[Theorem 4.7]{CKLT}) Let $\A\subseteq\BK$ be a von Neumann algebra and $\mc{E}:\A\to\BH$ be a quantum channel. A von Neumann subalgebra $\B\subseteq\BH$ is correctable (respectively, private) for $\mc{E}$ if and only if it is private (respectively, correctable) for any complement $\mc{E}^c$. In particular, if $\B$ is private for $\mc{E}$ and $\mc{E}^c$ is a minimal complement of $\mc{E}$ defined relative to a minimal Stinespring representation $(\pi,V,\mc{L})$, then there exists a normal unital $*$-homomorphism $\mc{R}:\B\to\pi(\A)'$ such that 
$$\mc{R}(b)V=Vb, \ \ \ b\in\B.$$
\end{thm}

\section{Operator recovery through Connes cocycle flow}

In this section, we establish our exact (\S3.1) and approximate (\S3.3) complementary recovery results, and discuss emergent type $\mathrm{III}_1$ structures (\S3.2).

\subsection{Exact recovery}
\label{s:exact}

In this section we establish a general complementary recovery result for von Neumann algebras, which, in the context of exact holographic error correction, connects holographic relative entropy \cite{KK}, holographic conditional expectations \cite{Faulkner:2020hzi} and the preservation of Connes' cocycle flow. In essence, it is an application of the theory of sufficient subalgebras (see, e.g., \cite{OP,Petz}), together with standard techniques in modular theory. In Section \ref{s:4}, we show that Theorem \ref{t:main} below applies to the case of Klein--Gordon fields in (the universal cover of) anti-de Sitter space where $\A$ and $\B$ (in Theorem \ref{t:main}) represent the observable algebras of boundary causal diamonds and their associated bulk causal wedges, respectively.

\begin{t:main} 
Let $V:\cH\hookrightarrow\K$ be an isometry between Hilbert spaces $\cH$ and $\K$. Let $\A$ and $\B$ be von Neumann algebras on $\K$ and $\cH$, respectively. Suppose there is a cyclic and separating state vector $\Omega\in\cH$ for $\B$ such that $V\Omega\in\K$ is cyclic and separating for $\A$. Then the following conditions are equivalent:
\begin{enumerate}
\item (Complementary Recovery) $\B$ and $\B'$ are correctable subalgebras for the complementary channels $V^*(\cdot)V:\A\to\BH$ and $V^*(\cdot)V:\A'\to\BH$, respectively.
\item (Preservation of CC Flow) For all states $\psi\in\cH$, we have $$V[D\om_{\Om}:D\om_\psi]_t=[D\om_{V\Om}:D\om_{V\psi}]_tV, \ \  V[D\om_{\Om}':D\om_{\psi}']_t=[D\om_{V\Om}':D\om_{V\psi}']_tV, \ \ \ t\in\R.$$
\item (Preservation of Relative Entropy) For all states $\psi\in\cH$, 
$$S_{\A}(\om_{V\psi},\om_{V\Om})=S_{\B}(\om_{\psi},\om_{\Om}) \ \ \ S_{\A'}(\om_{V\psi}',\om_{V\Om}')=S_{\B'}(\om_{\psi}',\om_{\vphi}').$$
\end{enumerate}
\end{t:main}

\begin{proof} $\bm{(1)\Rightarrow(2)}$: Let $\mc{E}$ denote the channel
$$\mc{E}:\BK\ni T\mapsto V^*TV\in\BH .$$
Since $V\Om$ is cyclic for both $\A$ and $\A'$, it follows that $(\id,V,\mc{K})$ is a minimal Stinespring representation for both $\mc{E}|_{\A}$ and $\mc{E}|_{\A'}$. 

Since $\B$ and $\B'$ are correctable for $\mc{E}|_{\A}$ and $\mc{E}|_{\A'}$, respectively, $\B$ and $\B'$ are private for the minimal complements $\mc{E}|_{\A'}=(\mc{E}|_{\A})^c$ and   $\mc{E}|_{\A}=(\mc{E}|_{\A'})^c$, respectively. By Theorem \ref{t:comp} there exist normal unital $*$-homomorphisms $\mc{R}:\B\to\A$ and $\mc{R}':\B'\to\A'$ such that 
\begin{equation}\label{e:Arv}\mc{R}(b)V=Vb, \ \  \mc{R}'(b')V=Vb', \ \ \  b\in\B, \ b'\in\B'.\end{equation}
Moreover, privacy (by definition) implies that $\mc{E}(\A)\subseteq\B$ and $\mc{E}(\A')\subseteq\B'$ (facts which can also be verified explicitly using (\ref{e:Arv}). Noting that the density of $\A'V\Omega$ in $\cH$ implies that $\mc{E}|_{\A}$ is faithful, it follows that 
$$E:=\mc{R}\circ\mc{E}|_{\A}:\A\to\A$$
is a normal faithful conditional expectation from $\A$ onto the (unital) subalgebra $\mc{R}(\B)$ (as noted in \cite{Faulkner:2020hzi}). Moreover, by (\ref{e:Arv}) for every $\psi\in\cH$, we have
$$\om_{V\psi}\circ E(a)=\la V\psi,\mc{R}(\mc{E}(a))V\psi\ra=\la\psi,\mc{E}(a)\psi\ra=\la V\psi,aV\psi\ra=\om_{V\psi}(a),$$
\emph{i.e.}, $E$ preserves all vector states of the form $\om_{V\psi}$, $\psi\in\cH$. It follows that $\mc{R}(\B)$ is a sufficient subalgebra of $\A$ for the collection $(\om_{V\psi})$, $\psi\in\cH$. By \cite[Theorem 9.3]{OP}, for all $\B$-separating vectors $\psi\in\cH$ we have $[D\om_{V\Omega}:D\om_{V\psi}]_t\in\mc{R}(\B)$ for all $t\in\R$ (note that $V\psi$ is separating for $\A$ by faithfulness of $\mc{E}$ and $\om_\psi|_{\B}$). 

Let $u_t\in\B$ be such that $\mc{R}(u_t)=[D\om_{V\Om}:D\om_{V\psi}]_t$. Since $\mc{R}$ is a unital injective $*$-homomorphism, each $u_t$ is necessarily unitary. We first show that $u_t=[D\om_{\Om},D\om_{\psi}]_t$ using uniqueness of the KMS condition (4) of Theorem \ref{thm:T2}.

As $u_t=V^*([D\om_{V\Om}:D\om_{V\psi}]_t)V$, $t\in\R$, we know that the family $(u_t)$ is $\sigma$-strongly continuous. From \cite[Theorem 1.1(4)]{Gesteau:2020rtg} (or \cite[Theorem 2]{Faulkner:2020hzi}), it follows that 
$$\mc{R}(\sigma^{\psi}_t(b))=\sigma^{V\psi}_t(\mc{R}(b)),\quad b\in\B ,\quad t\in\R.$$ 
We include the details of the convenience of the reader: by Takesaki's theorem \cite{Takesaki72} applied to the invariance $\om_{V\psi}\circ E=\om_{V\psi}$, it follows that the modular automorphism group $(\sigma^{V\psi}_t)$ leaves the subalgebra $\mc{R}(\B)$ invariant. Then, by \cite[Corollary VIII.1.4]{T2}, which follows from uniqueness of the KMS condition for modular flow, we have 
$$\sigma^\psi_t=\mc{R}^{-1}\circ\sigma^{V\psi}_t\circ\mc{R},\quad t\in\R.$$
By the KMS condition (\ref{thm:T2}-(4)) of $[D\om_{V\Om}:D\om_{V\psi}]_t$, given $b,c\in\B$, there is a bounded continuous function $F$ on the closed horizontal strip $\overline{\mathbb{S}}$  which is holomorphic on the open strip $\mathbb{S}$ such that 
\begin{align*}
    F\left( t\right) &=\om_{V\Omega} \left( [D\om_{V\Om}:D\om_{V\psi}]_t\sigma_t^{V\psi}(\mc{R}(c))\mc{R}(b)\right)
    =\om_{V\Omega} \left( \mc{R}(u_t\sigma_t^{\psi}(c)b)\right)
    =\om_{\Omega} \left( u_t\sigma_t^{\psi}(c)b\right) ,\\
    F\left( t+i\right) &=\om_{V\psi} \left( \mc{R}(b)[D\om_{V\Om}:D\om_{V\psi}]_t\sigma^{V\psi}(\mc{R}(c))\right)
    =\om_{V\psi} \left(\mc{R}(bu_t\sigma^{\psi}_t(c))\right)
    =\om_{\psi} \left( bu_t\sigma^{\psi}_t(c)\right) ,
\end{align*}
for every $t\in\R$.
By uniqueness of the KMS condition,  $u_t=[D\om_{\Om},D\om_{\psi}]_t$, $t\in\R$. 

We now deduce the general case of $\psi\in\cH$, $\norm{\psi}=1$ from the separating case. Recall that the support projection $s(\om_{\psi})$ of $\om_\psi\in \B_*$ satisfies $s(\om_{\psi})\cH=\overline{\B'\psi}$. Similarly, $s(\om_{V\psi})\K=\overline{\A'V\psi}$. 
Then, for any $\xi\in\mc{H}$,
$$Vs(\om_{\psi})\xi=\lim_{n\to\infty}Vb_n'\psi=\lim_{n\to\infty}\mc{R}'(b'_n)V\psi=\lim_{n\to\infty}s(\om_{V\psi})\mc{R}'(b'_n)V\psi=s(\om_{V\psi})Vs(\om_{\psi})\xi.$$
But for $\eta\in\B'\psi^{\perp}$ and $\zeta\in\K$, we have
$$\la s(\om_{V\psi})V\eta,\zeta\ra=\la V\eta,s(\om_{V\psi})\zeta\ra=\lim_{n\to\infty}\la V\eta,a_n'V\psi\ra=\lim_{n\to\infty}\la \eta,\underbrace{\mc{E}(a_n')}_{\in\B'}\psi\ra=0.$$
It follows that 
\begin{equation}\label{e:supp}
Vs(\om_{\psi})=s(\om_{V\psi})V.
\end{equation}

Next, we fix a faithful normal state $\om_0'$ on $\B'$. Pick a normal faithful state $\omega$ on $(1-s(\om_\psi))\B(1-s(\om_\psi))$, and consider 
$$\tilde{\omega}=2^{-1}(\om_\psi+\omega).$$
Then, $\tilde{\omega}$ is a normal faithful state on $\B$. By \cite[equations (4.5), (4.6)]{OP} (see property (2) of spatial derivatives in section \ref{s:modular}) 
$$\Delta(\om_\psi/\om_0')=\Delta((\om_\psi+\om)/\om_0')s(\om_\psi)=2\Delta(\tilde{\omega}/\om_0')s(\om_\psi).$$
Thus, $\Delta(\tilde{\omega}/\om_0')s(\om_\psi)$ is a self-adjoint operator, which implies that $[\Delta(\tilde{\omega}/\om_0'),s(\om_\psi)]=0$ on $\mc{D}(\Delta(\om_\psi/\om_0'))$. It follows that for any $t\in\R$, 
$$\Delta(\om_\psi/\om_0')^{it}=2^{it}\Delta(\tilde{\om}/\om_0')^{it}s(\om_\psi).$$ 
Thus we find that
\begin{align*}
[D\om_{\Om}:D\om_{\psi}]_t&=\Delta(\om_{\Om}/\om_0')^{it}\Delta(\om_\psi/\om_0')^{-it}\\
&=2^{-it}\Delta(\om_{\Om}/\om_0')^{it}\Delta(\tilde{\om}/\om_0')^{-it}s(\om_\psi)\\
&=2^{-it}[D\om_{\Om}:D\tilde{\omega}]_ts(\om_\psi).\end{align*}

Since $\B$ is necessarily in standard form on $\cH$ (as it contains a cyclic and separating vector), we can write the normal faithful state $\tilde{\om}=\om_{\xi_{\tilde{\om}}}|_{\B}$ for a separating vector $\xi_{\tilde{\om}}$. By the separating property and equation \eqref{e:supp}, we have
\begin{align*}V[D\om_{\Om}:D\om_{\psi}]_t&=2^{-it}V[D\om_{\Om}:D\om_{\xi_{\tilde{\om}}}]_ts(\om_\psi)\\
&=2^{-it}[D\om_{V\Om}:D\om_{V\xi_{\tilde{\om}}}]_tVs(\om_\psi)\\
&=2^{-it}[D\om_{V\Om}:D\om_{V\xi_{\tilde{\om}}}]_ts(\om_{V\psi})V.
\end{align*}
But, on $\A$, we find that
$$\om_{V\xi_{\tilde{\om}}}=\om_{\xi_{\tilde{\om}}}\circ\mc{E}=\tilde{\om}\circ\mc{E}=2^{-1}(\om_{V\psi}+\om\circ\mc{E}),$$
and by using the support projection $s(\om)=1-s(\om_{\psi})$ and equation \eqref{e:supp}, we have 
$$\om\circ\mc{E}(1-s(\om_{V\psi}))=\om(1-s(\om_{\psi}))=1.$$
Thus, 
$$s(\om\circ\mc{E})\leq 1-s(\om_{V\psi}).$$ 
For any normal faithful state $\vphi'\in\A'$, it follows from the properties (1) and (2) of spatial derivatives that 
$$\Delta(\om_{V\xi_{\tilde{\om}}}/\vphi')s(\om_{V\psi})=\frac{1}{2}\Delta(\om_{V\psi}/\vphi')s(\om_{V\psi})=\frac{1}{2}\Delta(\om_{V\psi}/\vphi').$$
Hence, we get
\begin{align*}
[D\om_{V\Om}:D\om_{V\xi_{\tilde{\om}}}]_ts(\om_{V\psi})&=\Delta(\om_{V\Om}/\vphi')^{-it}\Delta(\om_{V\xi_{\tilde{\om}}}/\vphi')^{-it}s(\om_{V\psi})\\
&=2^{it}\Delta(\om_{V\Om}/\vphi')^{it}\Delta(\om_{V\psi}/\vphi')^{-it}\\
&=2^{it}[D\om_{V\Om}:D\om_{V\psi}]_t.
\end{align*}
Putting this together, we finally see that
$$V[D\om_{\Om}:D\om_{\psi}]_t=2^{-it}[D\om_{V\Om}:D\om_{V\xi_{\tilde{\om}}}]_ts(\om_{V\psi})V=[D\om_{V\Om}:D\om_{V\psi}]_tV, \ \ \ t\in\R.$$

By symmetry, the same argument shows that 
$$V[D\om_{\Om}':D\om_{\psi}']_t=[D\om_{V\Om}':D\om_{V\psi}']_tV, \ \ \ t\in\R, \ \psi\in\cH,$$
where $[D\om_{\Om}':D\om_{\psi}']_t$ and $[D\om_{V\Om}':D\om_{V\psi}']_t$ are the cocycle flows on $\B'$ and $\A'$, respectively.

$\bm{(2)\Rightarrow(3)}$: Let $\psi\in\cH$ be a state. Pick a normal faithful functional $\om'$ on $(1-s(\om'_\psi))\B'(1-s(\om'_\psi))$, so that $\om_0'=\om'_\psi+\om'$ is faithful on $\B'$. Then, by properties of spatial derivatives
$$\Delta(\om_0'/\om_{\psi})\psi=\Delta(\om_0'/\om_{\psi})s(\om'_{\psi})\psi=\Delta(\om'_\psi/\om_\psi)\psi=\psi,$$
we find that
$$\Delta(\om_{\psi}/\om_0')^{-it}\psi=\Delta(\om_0'/\om_{\psi})^{it}\psi=\psi, \ \ \ t\in\R.$$ 
Combined with similar support arguments used above, we have
\begin{align*}[D\om_{\Om}:D\om_{\psi}]_t\psi&=\Delta(\om_{\Om}/\om_0')^{it}\Delta(\om_{\psi}/\om_0')^{-it}\psi\\
&=\Delta(\om_{\Om}/\om_0')^{it}\psi\\
&=\Delta(\om_0'/\om_{\Om})^{-it}\psi\\
&=\Delta(\om'_\psi/\om_{\Om})^{-it}\psi\\
&=\Delta(\om_{\Om}/\om'_\psi)^{it}\psi.
\end{align*}
Then the condition (2) implies that
\begin{equation}\label{e:V}
V\Delta(\om_{\Om}/\om'_\psi)^{it}\psi=V[D\om_{\Om}:D\om_{\psi}]_t\psi=[D\om_{V\Om}:D\om_{V\psi}]_tV\psi=\Delta(\om_{V\Om}/\om'_{V\psi})^{it}V\psi.
\end{equation}

Since $\psi\in\mc{D}(\Delta(\om_{\Om}/\om'_{\psi})^{1/2})$, there exists a unique bounded strongly continuous function \
$f_\psi:\overline{\mathbb{S}}_{-1/2}\to s(\om'_{\psi})\cH\subseteq\cH$, such that 
$$f_\psi(t)=\Delta(\om_{\Om}/\om'_{\psi})^{it}\psi, \ \ \ t\in\R,$$
and $f_\psi$ is strongly analytic on $\mathbb{S}_{-1/2}=\{z\in\C\mid -1/2<\mathrm{Im}(z)<0\}$ (see, e.g., \cite[Theorem A.7]{Hiai}). Similarly, there exists a unique bounded strongly continuous function $f_{V\psi}:\overline{\mathbb{S}}_{-1/2}\to s(\om'_{V\psi})\K\subseteq\K$
such that 
$$f_{V\psi}(t)=\Delta(\om_{V\Om}/\om'_{V\psi})^{it}V\psi$$ 
and $f_{V\psi}$ is strongly analytic on $\mathbb{S}_{-1/2}$. Then, $Vf_\psi:\overline{\mathbb{S}}_{-1/2}\to\K$ is bounded, strongly continuous, and analytic on $\mathbb{S}_{-1/2}$ with $Vf_\psi(t)=f_{V\psi}(t)$ for all $t\in\R$. Due to the uniqueness of analytic continuations, we find that for all $z\in \overline{\mathbb{S}}_{-1/2}$,
$$f_{V\psi}(z)=Vf_{\psi}(z).$$
In particular, for $t\in(0,1)$ ,
$$V\Delta(\om_{\Om}/\om_{\psi})^{t/2}\psi=Vf_{\psi}(-it/2)=f_{V\psi}(-it/2)=\Delta(\om_{V\Om}/\om_{V\psi})^{t/2}V\psi.$$
Then, equation \eqref{e:log} implies that
\begin{align*}S_{\A}(\om_{V\psi},\om_{V\Om})&=-\lim_{t\to0^+}t^{-1}(\norm{\Delta(\om_{V\Om}/\om'_{V\psi})^{t/2}V\psi}^2-1)\\
&=-\lim_{t\to0^+}t^{-1}(\norm{V\Delta(\om_{\Om}/\om'_{\psi})^{t/2}\psi}^2-1)\\
&=-\lim_{t\to0^+}t^{-1}(\norm{\Delta(\om_{\Om}/\om'_{\psi})^{t/2}\psi}^2-1)\\
&=S_{\B}(\om_{\psi},\om_{\Om}).
\end{align*}
The argument for $S_{\A'}(\om'_{V\psi},\om'_{V\Om})=S_{\B'}(\om'_{\psi},\om'_{\Om})$ is identical.

$\bm{(3)\Rightarrow(1)}$: This implication (and the equivalence of (1) and (3)) was shown in \cite[Theorem 1.1]{KK} under the a priori additional assumption that the set of cyclic and separating vectors for $\B$ is dense in $\cH$. This assumption, however, follows since we have a single cyclic and separating vector $\Om\in\cH$ for $\B$. Indeed, if $a\in\B'$ is invertible, it follows that $a\Om$ is also cyclic and separating for $\B$. Since the invertible elements of $\B'$ are strongly dense in $\B'$ \cite[Proposition 1]{DM} and $\B'\Om$ is dense in $\cH$, it follows that $\mathrm{Inv}(\B')\Om$ is dense in $\cH$. Combined with \cite[Theorem 1.1]{KK}, the implication holds. 
\end{proof}




\begin{rem}\label{r:r1}
Given $\psi\in\cH$ such that $S_{\B}(\om_{\psi},\om_{\Om})<\infty$, by monotonicity of the relative entropy \cite[Theorem 5.3]{OP}
$$S_{\A}(\om_{V\psi},\om_{V\Om})=S_{\A}(\om_{\psi}\circ\mc{E},\om_{\Om}\circ\mc{E})\leq S_{\B}(\om_{\psi},\om_{\Om})<\infty.$$
Assuming condition (2) in Theorem \ref{t:main}, the derivative formula for the relative entropy \cite[Theorem 5.7]{OP} then gives
\begin{align*}S_{\A}(\om_{V\psi},\om_{V\Om})&=i\lim_{s\to\infty}s^{-1}(\om_{V\psi}((D\om_{V\Om}:D\om_{V\psi})_t)-1)\\
&=i\lim_{s\to\infty}s^{-1}(\om_{\psi}((D\om_{\Omega}:D\om_{\psi})_t)-1)\\
&=S_{\B}(\om_{\psi},\om_{\Om}).
\end{align*}
However, when $S_{\B}(\om_{\psi},\om_{\Om})=\infty$, we cannot appeal directly to the derivative formula. The analytic continuation argument above circumnavigates this, and establishes the equality of relative entropy even when $S_{\B}(\om_{\psi},\om_{\Om})=\infty$.
\end{rem}


\subsection{Emergent type III$_1$ structures}
\label{s:III1}

Suppose, in addition to the hypotheses of Theorem \ref{t:main}, that $\A$ is a type III$_1$ factor and that $V\Omega$ is ergodic in the sense that the fixed point algebra is 
\begin{align}
    \mc{A}^{V\Omega}=\{a\in \A\mid \sigma^{V\Omega}_t(a)=a, \ t\in\R\}=\C1\,.
\end{align} 
By the proof of Theorem \ref{t:main} (see also, \cite[Theorem 2(c)]{Faulkner:2020hzi}), the recovery map intertwines the modular flows: 
\begin{align}
    \sigma^{V\Omega}_t(\mc{R}(b))=\mc{R}(\sigma^{\Omega}_t(b)),\quad b\in\mc{B}. 
\end{align}
Hence, if $\sigma^{\Omega}_t(b)=b$ for all $s\in\R$, then $\mc{R}(b)\in\mc{A}^{V\Omega}=\C1$. By injectivity of $\mc{R}$, we must have $b\in\C1$, implying that $\Omega$ is an ergodic state for $\B$. By the proof of \cite[Theorem 3]{Longo79}, it follows that either $\mc{B}=\C1$, in which case $\mathrm{dim}(\cH)=1$, or $\mc{B}$ is a type III$_1$ factor. This observation supports recent literature on the emergence of type III$_1$ structures in conformal field theory and holography \cite{Gabbiani:1992ar,Yngvason:2004uh,Kang:2019dfi,Leutheusser:2021qhd}.

\subsection{Approximate recovery}
\label{s:approx}

In this subsection we investigate an approximate version of Theorem \ref{t:main}. Results of a similar nature were recently established for hyperfinite bulk algebras by Faulkner and Li \cite[Theorem 2]{FL}. Indeed, we rely on their result \cite[Theorem 1]{FL} for the implication $(3)\Rightarrow(1)$ in the hyperfinite case of Theorem \ref{t:approx} below. Our contribution here is to show that, for general bulk observable algebras, approximate intertwining of cocycle derivatives is necessary for approximate bulk recovery, and sufficient for approximate privacy (the latter being equivalent to approximate recovery in the hyperfinite case by \cite[Theorem 1]{FL}). We also employ different techniques than Faulkner--Li \cite{FL}, e.g., formulating approximate error correction as exact error correction in an ultraproduct of von Neumann algebras. See \cite{AH,Raynaud} for details on ultraproducts.

We require the following (known) fact, whose proof we include for convenience of the reader.

\begin{lem}\label{l:cocycle} Let $\B\subseteq\BH$ be a von Neumann algebra and $\Omega\in\cH$ be a cyclic and separating state vector. Then $\B$ coincides with the von Neumann algebra generated by 
\begin{equation}\label{e:set}
\{[D\om_{\psi}: D\om_{\Omega}]_t\mid t\in\R, \textnormal{$\psi$ cyclic and separating vector}\}.\end{equation}
\end{lem}

\begin{proof} Let $\B_\Omega$ denote the von Neumann algebra generated by equation \eqref{e:set}. By the cocycle property, for each $s\in\R$, and cyclic and separating $\psi$, 
$$\sigma^\Omega_s([D\om_{\psi}: D\om_{\Omega}]_t)=[D\om_{\psi}:D\om_{\Omega}]_s^*[D\om_{\psi}: D\om_{\Omega}]_{s+t}\in\B_{\Omega}.$$ 
Thus, $\B_\Omega$ is stable under the modular automorphism group of $\om_{\Omega}$, implying the existence of a unique normal faithful $\om_{\Omega}$-preserving conditional expectation $E:\B\to\B_{\Omega}$ \cite{Takesaki72}. 
Let $\om_0=\omega_{\Om}|_{\B_\Omega}$.
Since $([D\om_{\psi}: D\om_{\Omega}]_t)_{t\in\R}\subseteq \B_{\Omega}$ is a cocycle with respect to $\sigma^{\omega_0}$, by Connes' theorem \cite[Th\'{e}or\`{e}me 1.2.4]{Connes73}, there exists a faithful normal weight $\vphi$ on $\B_\Omega$ such that 
$$[D\vphi:D\om_0]_t=[D\om_{\psi}: D\om_{\Omega}]_t, \ \ \ t\in\R.$$
(see \cite[\S VIII.3]{T2} for cocycle derivatives of weights). But then 
$$[D(\vphi\circ E):D\om_{\Omega}]_t=[D(\vphi\circ E):D(\om_{\Omega}\circ E)]_t=[D\vphi:D\om_0]_t=[D\om_{\psi}: D\om_{\Omega}]_t \ \ \ t\in\R,$$
where the second equality follows from by \cite[Corollary IX.4.22(ii)]{T2}. Since $\om_\Omega$ is faithful, by uniqueness of cocyle derivatives (which follows from uniqueness of spatial derivatives \cite[Proposition IX.3.10(i)]{T2}), we have $\vphi\circ E=\om_{\psi}$, implying $\om_{\psi}\circ E=\om_{\psi}$. Since the set of cyclic and separating vectors is dense in $\cH$ (see the proof of $(3)\Rightarrow(1)$ in Theorem \ref{t:main}), it follows that $E=\id_{\B}$, and $\B=\B_\Omega$.
\end{proof}

\begin{t:approx} 
Let $(V_n:\cH\hookrightarrow\K_n)_{n\in\N}$ be a sequence of isometries between Hilbert spaces $\cH$ and $\K_n$. Let $\A_n$ and $\B$ be von Neumann algebras on $\K_n$ and $\cH$, respectively. Suppose there is a cyclic and separating state vector $\Omega\in\cH$ for $\B$ such that $V_n\Omega\in\K_n$ is cyclic and separating for each $\A_n$. Consider the following conditions:
\begin{enumerate}
\item There exist sequences $(R_n:\B\to\A_n)_{n\in \N}, (R_n':\B'\to\A_n')_{n\in \N}$ of channels such that
$$wot-\lim_nV_n^*R_n(b)V_n=b, \ \ wot-\lim V_n^*R'_n(b')V_n=b'$$
for all $b\in\B$ and $b'\in\B'$.
\item For all cyclic and separating states $\psi\in\cH$ (for $\B$), we have 
\begin{align*}
sot-\lim_nV_n[D\om_{\Om}:D\om_\psi]_t&-[D\om_{V_n\Om}:D\om_{V_n\psi}]_tV_n=0\\
sot-\lim_nV_n[D\om_{\Om}':D\om_{\psi}']_t&-[D\om_{V_n\Om}':D\om_{V_n\psi}']_tV_n=0, \ \ \ t\in\R.
\end{align*}
\item For all bounded sequences $(a_n)\in\prod_{n\in\N}\A_n$, $(a'_n)\in\prod_{n\in\N}\A'_n$, and  and $b\in\B$, $b'\in\B$,
$$wot-\lim_n[V_n^*a_nV_n,b']=0, \ \ \ wot-\lim_n[V_n^*a'_nV_n,b]=0.$$
\end{enumerate}
Then $(1)\Rightarrow(2)\Rightarrow(3)$, and if $\B$ is hyperfinite, the three conditions are equivalent. 
\end{t:approx}

\begin{rem} The (approximate) theory of sufficient subalgebras \cite{Petz} (see also \cite[\S9]{OP}), which addresses similar questions of approximate recovery, does not directly apply here as the ranges $V_n^*\A_n V_n$ are not necessarily contained in $\B$. This lack of range control renders the maps $V_n^*(\cdot)V_n$ incompatible with some of the modular techniques in the proof of \cite[Lemma 3.1]{Petz}.
\end{rem}

\begin{proof}[Proof of Theorem \ref{t:approx}] $(1)\Rightarrow (2)$: First pass to an arbitrary subnet of $(V_{n_i})_{i\in I}$ of $(V_n)_{n\in\N}$. Note that by direct computation, for any unitaries $u\in\mc{U}(\B)$ and $u'\in\mc{U}(\B')$,
\begin{equation}\label{e:sot}sot-\lim_{i\in I} R_{n_i}(u)V_{n_i}-V_{n_i}u=0, \ \ sot-\lim_{i\in I} R'_{n_i}(u')V_{n_i}-V_{n_i}u'=0.\end{equation}

Let $\mc{U}$ be an ultrafilter on $I$ which dominates the order (i.e., tails) filter. Since $I\ni i\mapsto n_i \in\N$ is monotonic and co-final, and $\N$ has no greatest element, the order filter is necessarily free, implying that $\mc{U}$ is a free ultrafilter.


Since $V_{n_i}\Omega$ is cyclic and separating for $\A_{n_i}$, the action of $\A_{n_i}$ on $\K_{n_i}$ is standard. Let 
$$\A=\left(\prod_{\mc{U}}\A_{n_i}\right)''$$
denote the Groh--Raynaud ultraproduct of the standard von Neumann algebras $\A_i$ \cite{Groh,Raynaud}, which is defined as the strong operator closure of the Banach space ultraproduct $\prod_{\mc{U}}\A_{n_i}$ inside $\mc{B}(\K)$, where 
$$\K=\prod_{\mc{U}}\K_{n_i}$$ 
is the ultraproduct of the Hilbert spaces $\K_{n_i}$, and the action of $\prod_{\mc{U}}\A_{n_i}$ on $\K$ is 
$$(a_{n_i})_{\mc{U}}(\psi_{n_i})_{\mc{U}}=(a_{n_i}\psi_{n_i})_{\mc{U}}.$$
Note that the action of each $\A_{n_i}$ on $\K_{n_i}$ is standard as $V_{n_i}\Omega$ is cyclic and separating.
The map $V:\cH\to \K$ given by $V\psi=(V_{n_i}\psi)_{\mc{U}}$ is an isometry and 
$$\mc{E}:\mc{B}(\K)\to\mc{B}(\cH)\,,\quad\mc{E}=V^*(\cdot)V$$
is a quantum channel. If $(a_{n_i})_{\mc{U}}\in\prod_{\mc{U}}\A_{n_i}$, $u'\in\mc{U}(\B')$, and $\xi,\eta\in\cH$, we have
\begin{align*}
\la V^*((a_{n_i})_{\mc{U}})Vu'\xi,\eta\ra&=\lim_{i\to\mc{U}}\la a_{n_i}V_{n_i}u'\xi,V_{n_i}\eta\ra\\
&=\lim_{i\to\mc{U}}\la a_{n_i}R_{n_i}(u')V_{n_i}\xi,V_{n_i}\eta\ra\\
&=\lim_{i\to\mc{U}}\la R_{n_i}(u')a_{n_i}V_{n_i}\xi,V_{n_i}\eta\ra\\
&=\lim_{i\to\mc{U}}\la a_{n_i}V_{n_i}\xi,R_{n_i}((u')^*)V_{n_i}\eta\ra\\
&=\lim_{i\to\mc{U}}\la a_{n_i}V_{n_i}\xi,V_{n_i}(u')^*\eta\ra\\
&=\la V^*(a_{n_i})_{\mc{U}}V\xi,(u')^*\eta\ra\\
&=\la u'V^*(a_{n_i})_{\mc{U}}V\xi,\eta\ra.
\end{align*}
The first equality follows by definition of the ultraproduct Hilbert space, the second from equation \eqref{e:sot} and the boundedness of $(V_{n_i}\eta)$, the third from $R_{n_i}'(\B')\subseteq\A_{n_i}'$, and the fifth from equation \eqref{e:sot} and the boundedness of $(a_{n_i}V_{n_i}\xi)$. It follows by normality of $\mc{E}$ that 
$$\mc{E}(\A)\subseteq\B.$$ In other words, $\B'$ is private for $\mc{E}|_{\A}$.

As shown in \cite[Theorem 1.8]{Raynaud}, $\A'=(\prod_{i\in\mc{U}}\A_{n_i}')''$ (in the case when all $\A_{n_i}$ are equal; see \cite[Theorem 3.22]{AH} which handles the case when $I=\mathbb{N}$, and the $\A_{n_i}$ are not necessarily equal; the result is true in general). Hence, by symmetry of condition (1), we have 
$$\mc{E}(\A')\subseteq\B',$$ 
i.e., $\B$ is private for $\mc{E}|_{\A'}$. By the privacy/correctability duality \cite[Theorem 4.7]{CKLT}, we have recovery channels $\B\to\A$ and $\B'\to\A'$; however, we are not quite in position to apply Theorem~\ref{t:main} as $V\Omega$ is not cyclic and separating for $\A$. Thus, we need to work in a particular ``corner'' of $\A$. 

Let $p\in\A$ denote the support projection of the vector state $\om_{V\Omega}|_{\A}=\om_{(V_{n_i}\Omega)_{\mc{U}}}|_{\A}$.
Then $V^*(1-p)V\in\B$ (by above) and $\la\Omega,V^*(1-p)V\Omega\ra=0$. Since $\Omega$ is separating for $\B$, we have $V^*(1-p)V=0$, from which it follows that $pV=V$. 

It is known that $\A$ is standardly represented on $\K$ and that the associated conjugate linear isometry $J=(J_{n_i})_{\mc{U}}$, where $J_{n_i}=J_{V_{n_i}\Omega,\A_{n_i}}$ \cite[Corollary 3.7]{Raynaud} (see also \cite[Theorem 3.18]{AH}). Then, $JpJ$ is the support projection of $\om_{V\Omega}|_{\A'}$, and by symmetry we have $JpJV=V$. Letting $q=pJpJ$, by the general theory of standard forms, $q\A q$ is standardly represented on $q\K$, with $J_{q\A q}=qJq$ and $pap\mapsto qaq$ is a $*$-isomorphism $p\A p\cong q\A q$ \cite[Corollary 2.5, Lemma 2.6]{Haagerup75}. As $qV=V$, the vector $V\Omega\in q\K$ is separating for both $q\A q$ and $(q\A q)'=q\A' q$ (see \cite[Lemma 2.4]{Haagerup75} for the latter equality), and hence cyclic and separating for $q\A q$. The invariance $qV=V$ and the above calculations also show that 
$$\mc{E}(q\A q)\subseteq\B, \quad \mc{E}(q\A'q)\subseteq\B'.$$ Now combining Theorem~\ref{t:comp} with Theorem~\ref{t:main}, we have that for all $\psi\in\cH$,
\begin{align}
\begin{aligned}\label{e:exact}
    V[D\om_{\Om}:D\om_\psi]_t &=[D\om_{V\Om}:D\om_{V\psi}]_tV\,, \\  V[D\om_{\Om}':D\om_{\psi}']_t &=[D\om_{V\Om}':D\om_{V\psi}']_tV\,, \quad t\in\R.
\end{aligned}
\end{align}
If $\psi\in\cH$ is separating for $\B$ it follows that the support projection of $\om_{V\psi}|_{\A}$ is again $p$ (by the same argument for $\Omega$). Combining \cite[Proposition 3.15, Theorem 4.1]{AH} with the proof of \cite[Proposition 2.2(i)]{Raynaud}, it follows that
$$[D\om_{V\Om}:D\om_{V\psi}]^{p\A p}_t = p\left( [D\om_{V_{n_i}\Om}:D\om_{V_{n_i}\psi}]_t\right)_{\mc{U}}p.$$
Applying the $*$-isomorphism $p\A p \ni a\mapsto qaq\in q\A q$, we have
$$q[D\om_{V\Om}:D\om_{V\psi}]_tq = q[D\om_{V\Om}:D\om_{V\psi}]^{p\A p}_tq = q\left([D\om_{V_{n_i}\Om}:D\om_{V_{n_i}\psi}]_t\right)_{\mc{U}}q.$$
By equation \eqref{e:exact} and the invariance $qV=V$, for every $\xi,\eta\in\cH$, we therefore have
\begin{align*}
\la[D\om_{\Om}:D\om_\psi]_t\xi,\eta\ra&=\la V[D\om_{\Om}:D\om_\psi]_t\xi,V\eta\ra\\
&=\la [D\om_{V\Om}:D\om_{V\psi}]_tV\xi,V\eta\ra\\
&=\la q[D\om_{V\Om}:D\om_{V\psi}]_tqV\xi,V\eta\ra\\
&=\la ([D\om_{V_{n_i}\Om}:D\om_{V_{n_i}\psi}]_t)_{\mc{U}}V\xi,V\eta\ra\\
&=\la ([D\om_{V_{n_i}\Om}:D\om_{V_{n_i}\psi}]_tV_{n_i}\xi)_{\mc{U}},(V_{n_i}\eta)_{\mc{U}}\ra\\
&=\lim_{i\to\mc{U}}\la[D\om_{V_{n_i}\Om}:D\om_{V_{n_i}\psi}]_tV_{n_i}\xi,V_{n_i}\eta\ra\\
&=\lim_{i\to\mc{U}}\la V_{n_i}^*[D\om_{V_{n_i}\Om}:D\om_{V_{n_i}\psi}]_tV_{n_i}\xi,\eta\ra.
\end{align*}
Since $[D\om_{\Om}:D\om_\psi]_t$ is unitary and each $[D\om_{V_{n_i}\Om}:D\om_{V_{n_i}\psi}]_t$ is contractive, the standard argument shows that for all $\xi\in\cH$,
$$\lim_{i\to\mc{U}}\norm{V_{n_i}[D\om_{\Om}:D\om_\psi]_t\xi-[D\om_{V_{n_i}\Om}:D\om_{V_{n_i}\psi}]_tV_{n_i}\xi}=0,$$
that is, 
$$sot-\lim_{i\to\mc{U}}V_{n_i}[D\om_{\Om}:D\om_\psi]_t-[D\om_{V_{n_i}\Om}:D\om_{V_{n_i}\psi}]_tV_{n_i}=0.$$
It follows from the general correspondence between net and filter convergence (see, e.g., \cite[\S4]{Willard}) that for every separating $\psi\in\cH$, the net
$$\left( V_{n_i}[D\om_{\Om}:D\om_\psi]_t-[D\om_{V_{n_i}\Om}:D\om_{V_{n_i}\psi}]_tV_{n_i}\right)_{i\in I}$$
clusters to $0$ in the strong operator topology. Since the subnet $(n_i)_{i\in I}$ of the original sequence was arbitrary, it follows that the original sequence 
$$\left( V_{n}[D\om_{\Om}:D\om_\psi]_t-[D\om_{V_{n}\Om}:D\om_{V_{n}\psi}]_tV_{n}\right)_{n\in\N}$$
converges to zero in the strong operator topology. 

If $\psi\in\cH$ is in addition cyclic for $\B$, that is, separating for $\B'$, then the symmetry of equation \eqref{e:exact} implies the same result for the commutant cocycles. 

$(2)\Rightarrow(3)$ Boundedness of the sequence $(V_n)$ together with joint continuity of multiplication in the strong operator topology (on bounded sets) implies that 
$$[D\om_{\Om}:D\om_{\psi}]_t=sot-\lim_{n\to\infty}V_n^*[D\om_{V_n\Om}:D\om_{V_n\psi}]_t V_n,$$
for all $t\in\R$ and cyclic and separating separating $\psi\in\cH$. The convergence also holds in the weak operator topology. By continuity of the adjoint in the latter topology, we get
$$[D\om_{\psi}:D\om_{\Om}]_t=[D\om_{\Om}:D\om_{\psi}]_t^*=wot-\lim_{n\to\infty}V_n^*[D\om_{V_n\psi}:D\om_{V_n\Om}]_t V_n \ \ \ t\in\R.$$
Since $[D\om_{\psi}:D\om_{\Om}]_t$ is a unitary for every $t\in\R$, the standard argument shows that 
$$sot-\lim_{n\to\infty}V_n[D\om_{\psi}:D\om_{\Om}]_t-[D\om_{V_n\psi}:D\om_{V_n\Om}]_t V_n=0, \ \ \ t\in\R.$$
That is, the ($\B$ part of the) conclusion (2) also holds with the roles of $\Omega$ and $\psi$ reversed.

Now, fix a bounded sequence $(a_n')\in\prod_{n\in\N}\A_n'$, and $\xi,\eta\in\cH$. By boundedness and joint continuity of multiplication in the strong operator topology, for any cyclic and separating $\psi\in\cH$ and $t\in\R$, we therefore have
\begin{align*}\lim_{n\to\infty}\la V_n^*a'_n V_n[D\om_{\psi}:D\om_{\Om}]_t\xi,\eta\ra&=\lim_{n\to\infty}\la V_n^*a'_n[D\om_{V_n\psi}:D\om_{V_n\Om}]_t V_n\xi,\eta\ra\\
&=\lim_{n\to\infty}\la V_n^*[D\om_{V_n\psi}:D\om_{V_n\Om}]_ta_n'V_n\xi,\eta\ra\\
&=\lim_{n\to\infty}\la a_n'V_n\xi,[D\om_{V_n\Om}:D\om_{V_n\psi}]_tV_n\eta\ra\\
&=\lim_{n\to\infty}\la a_n'V_n\xi,V_n[D\om_{\Om}:D\om_{\psi}]_t\eta\ra\\
&=\lim_{n\to\infty}\la [D\om_{\psi}:D\om_{\Om}]_t V_n^*a'_n V_n\xi,\eta\ra.
\end{align*}
The same argument is valid with the roles of $\Omega$ and $\psi$ reversed, and also for finite linear combinations of finite products of $[D\om_{\psi}:D\om_{\Omega}]$ and $[D\om_{\Om}:D\om_{\psi}]$ for cyclic and separating $\psi\in\cH$. By boundedness, it follows that 
$$\lim_{n\to\infty}\la V_n^*a'_n V_nb_0\xi,\eta\ra=\lim_{n\to\infty}\la b_0 V_n^*a'_n V_n\xi,\eta\ra.$$
for all $b_0$ in the $C^*$-algebra $B_\Omega$ generated by 
$$\{[D\om_{\psi}: D\om_{\Omega}]_t\mid t\in\R, \textnormal{$\psi$ cyclic and separating vector}\}.$$

Let $b$ be a self-adjoint contraction in $\B$, which, by Lemma~\ref{l:cocycle} equals $B_\Omega''$. By Kaplansky's density theorem, we can approximate $b$ by a net $(b_i)$ of self-adjoint contractions from $B_\Omega$ in the strong operator topology. By boundedness of $(V_n^*a_n'V_n)$, it follows that
$$\lim_i\norm{V_n^*a_n'V_nb\xi-V_n^*a'_nV_nb_i\xi}=0$$
uniformly in $n\in\N$. Thus,
\begin{align*}
\lim_{n\to\infty}\la V_n^*a'_n V_nb\xi,\eta\ra&=\lim_i\lim_{n\to\infty}\la V_n^*a'_n V_nb_i\xi,\eta\ra\\
&=\lim_i\lim_{n\to\infty}\la b_iV_n^*a'_n V_n\xi,\eta\ra\\
&=\lim_i\lim_{n\to\infty}\la V_n^*a'_n V_n\xi,b_i\eta\ra\\
&=\lim_{n\to\infty}\la V_n^*a'_n V_n\xi,b\eta\ra\\
&=\lim_{n\to\infty}\la bV_n^*a'_n V_n\xi,\eta\ra .
\end{align*}
By linearity, it follows that 
$$\lim_{n\to\infty}\la V_n^*a'_n V_nb\xi,\eta\ra=\lim_{n\to\infty}\la bV_n^*a'_n V_n\xi,\eta\ra$$
for all $b\in\B$.

By symmetry, the same argument is valid for the commutant $\B'$, and we arrive at condition (3). This completes $(1)\Rightarrow(2)\Rightarrow(3)$.

Finally, when $\B$ is hyperfinite, $(3)\Rightarrow(1)$ holds by \cite[Theorem 1]{FL}. Therefore, all three conditions are are equivalent when $\B$ is hyperfinite.
\end{proof}

We note that approximate recovery studied in this section is more restrictive than that formulated in \cite{Gesteau:2021jzp}. Here, the Hilbert space corresponding to the bulk theory is fixed as $\mathcal{H}$, and there is a sequence of boundary Hilbert spaces $\mathcal{K}_n$. This way of building approximate recovery is based on taking a \emph{large $n$ limit}. In contrast, \cite{Gesteau:2021jzp} provides an explicit form of approximate recovery without needing a sequence of boundary Hilbert spaces. Topologically, the approximation of Theorem \ref{t:approx} (11) is with respect to the point weak* topology, whereas the approximation in \cite{Gesteau:2021jzp} is with respect to the completely bounded norm.

\section{Operator-algebraic subregion-subregion duality}\label{s:4}

In this section we give an operator algebraic proof of subregion-subregion duality between boundary causal diamonds and bulk causal wedges for (scalar) Klein--Gordon fields in the universal cover of AdS, thus giving a concrete instance of Theorem~\ref{t:main} in holography. The proof utilizes the recent framework of holography for Klein--Gordon fields in \cite{DW} together with the holographic error correction framework of \cite{Gesteau:2020rtg}. Combined with the known geometric modular structure for bulk wedges in vacuum AdS \cite{BEM} and general properties of cocycle derivatives, we obtain additional mathematical evidence in support of the kink transform conjecture of \cite{Bousso}.

The universal cover of $(d+1)$-dimensional anti-de Sitter space can be described as the manifold 
\begin{align}
    AdS_{d+1}=\R_t\times \overline{\mathbb{B}^d} ,
\end{align}
(where $\mathbb{B}^d$ is the open ball) with metric given by 
\begin{align}
    g=\frac{(1+z^2)dt^2-(1+z^2)^{-1}dz^2-d\om^2}{z^2}
\end{align}
on $[0, 1)_z\times\R_t\times \mathbb{S}_\om^{d-1}$ and with boundary $\partial AdS_{d+1} = \{z = 0\}$. Then $g$ is conformal to the metric $\tilde{g}:=z^2g$ on the interior $AdS_{d+1}^o$. As $z\to0$, we recover the cylindrical boundary metric:
\begin{align}
    \tilde{g}=(1+z^2)dt^2-\frac{1}{(1+z^2)}dz^2-d\om^2\quad \overset{z\to 0}{\loongrightarrow}\quad dt^2-d\om^2
\end{align}

The Klein--Gordon operator on $M:=(AdS_{d+1},g)$ is
\begin{align}
    P=\square_g+\nu^2-\frac{d^2}{4},
\end{align}
where $\square_g=|g|^{-1/2}\partial_\mu(|g|^{1/2}g^{\mu,\nu}\partial_\nu)$ is the Laplacian of $M$. Solutions to $Pu=f$ (on AdS proper and its universal cover) have been studied extensively from both physical and mathematical perspectives (see, e.g., \cite{AIS78,BZ82,BEM,DF16,DR03,F74,IW04,YG09}). Existence and uniqueness of solutions to $Pu=f$ on asymptotically AdS (aAdS) spacetimes with Dirichlet boundary conditions were studied recently in \cite{Vasy}, where it was shown that, with respect to certain Sobolev spaces, retarted/advanced propagators exist, leading to a natural symplectic space of solutions. (See, e.g., \cite{DF16,War13} for related results with other boundary conditions.)  Building on Vasy's work \cite{Vasy}, Wrochna \cite{Wrochna} introduced the notion of a holographic Hadamard state on aAdS spactimes, which was then utilized by Dybalski--Wrochna \cite{DW} towards a second quantization procedure analogous to the globally hyperbolic setting (see, e.g., \cite{G19}), as well as a mechanism for holography in aAdS spacetimes. For simplicity, in this paper we stick with $AdS_{d+1}$, but utilize the more general framework from \cite{DW} as in future work we hope to extend the analysis below to a class of aAdS spacetimes. See Section \ref{s:outlook} for a discussion.

\subsection{Symplectic solution space of Klein--Gordon equation}

We begin with a review of the constructions and notations from \cite{DW}, to which we refer the reader for details (see also \cite{Vasy,Wrochna}). Let $C^\infty(M)$ denote space of smooth functions on $M$ (that is, which posess smooth extensions across $\partial M$). The space of smooth functions which vanish along with all derivatives at $\partial M$ is denoted $\dot{C}^\infty(M)$. Its topological dual is denoted $C^{-\infty}(M)$. The related subspaces of compactly supported elements are respectively denoted $C^\infty_c(M)$, $\dot{C}^\infty_c(M)$ and $C^{-\infty}_c(M)$. On the boundaryless manifold $\partial M$, $\mc{D}'(\partial M)$ denotes the usual space of distributions (see, e.g., \cite[\S 6.3]{Horm}).

Let $\mathrm{Diff}_b(M)$ be the algebra of $b$-differential operators, those generated by smooth vector fields tangent to the boundary. 
Let $\mathrm{Diff}_0(M)$ be the space of smooth vector fields which vanish at the boundary. In what follows, we let $\la\cdot,\cdot\ra$ denote the canonical inner product on $L^2(M):=L^2(M,g)$ with respect to the associated volume form induced by $g$. For a nonnegative integer $k$, let
\begin{align}
    H_0^k(M):=\{u\in C^{-\infty}(M)\mid Qu\in L^2(M,g) \ \forall \ Q\in \mathrm{Diff}_0^k(M)\},
\end{align}
be the associated Sobolev space (topologized in the usual Hilbertian manner, see, e.g., \cite[\S2.4]{Wrochna} and the references therein). Above, $k$ denotes the order of a differential operator. $H_0^{-k}(M)$ is defined as the dual of $H_0^k(M)$ via the $L^2$-pairing, and we set $H_0^0(M)=L^2(M)$.

For $k\in\{-1,0,1\}$ and a nonnegative integer $s$, let
\begin{align}
    H^{k,s}_{0,b}(M):=\{u\in H^k_0(M)\mid Qu\in H^k_0(M) \ \forall \ Q\in \mathrm{Diff}^s_b(M)\},
\end{align}
This carries a natural Hilbert space topology (\cite[\S 2.4]{Wrochna}), and we let $H^{-k,-s}_{0,b}(M)$ denote its dual space via the $L^2$-pairing. Then,
\begin{align}
    H^{k,\infty}_{0,b}(M):=\bigcap_{s}H^{k,s}_{0,b}(M)
\end{align}
is the space of conormal distributions respective to $H^k_0(M)$, equipped with its canonical Fr\'{e}chet space topology. The space $H^{-k,-\infty}_{0,b}(M)$ is defined as the topological dual of $H^{k,\infty}_{0,b}(M)$. 

Let $H^{k,\pm\infty}_{0,b,c}(M)$ denote the subspaces of compactly supported elements, and let $H^{k,\pm\infty}_{o,b,loc}(M)$ denote the space of distributions whose localization with test functions (smoothly extendible across the boundary) belongs to $H^{k,\pm\infty}_{0,b}(M)$. Once appropriately topologized as inductive limits of Fr\'{e}chet spaces, $H^{k,\infty}_{0,b,loc}(M)$ is in duality with $H^{-k,-\infty}_{0,b,c}(M)$ and $H^{k,\infty}_{0,b,c}(M)$ is in duality with $H^{-k,-\infty}_{0,b,loc}(M)$. 

Over the interior $M^o$, $H^{k,\infty}_{0,b,loc}(M)$ coincides with $C^\infty(M^o)$ and $H^{k,-\infty}_{0,b,loc}(M)$ with $\mc{D}'(M^o)$. The spaces $H^{k,\infty}_{0,b,c}(M)$ provide suitable replacements of $C_c^\infty(M^o)$ in this context \cite{DW,Vasy,Wrochna}.

Denote the spaces of past/future supported elements of $H^{k,\infty}_{0,b,loc}(M)$ by
\begin{align}
    H^{k,\infty}_{0,b,\pm}(M):=\{u\in H^{k,\infty}_{0,b,loc}(M)\mid \mathrm{supp}(u)\subseteq\{\pm t\geq \pm t_0\}, \ \textnormal{some} \ t_0\in\R\}.
\end{align}
By \cite[Theorem 1.6]{Vasy}, there exist continuous operators 
\begin{align}
    P_{\pm}^{-1}:H^{-1,\infty}_{0,b,\pm}(M)\to H^{1,\infty}_{0,b,\pm}(M)
\end{align}
such that $PP_{\pm}^{-1}=\id$ on $H^{-1,\infty}_{0,b,\pm}(M)$ and $P_{\pm}^{-1}P=\id$ on $H^{1,\infty}_{0,b,\pm}(M)$. The difference
\begin{align}
    G:=P_+^{-1}-P_{-}^{-1}:H^{-1,\infty}_{0,b,c}(M)\to H^{1,\infty}_{0,b,loc}(M)
\end{align}
is the analogue of the causal propagator on globally hyperbolic spacetimes. By \cite[Proposition 3.4]{DW} (see also \cite[Proposition 3.1]{Wrochna}), the $\R$-bilinear form $\sigma:=\la \cdot,G\cdot\ra_{L^2(M)}$ induces a non-degenerate symplectic form on the (real) quotient space 
\begin{align}
    \mc{X}:=H^{-1,\infty}_{0,b,c}(M;\R)/PH^{1,\infty}_{0,b,c}(M;\R).
\end{align}
The symplectic space ($\mc{X},\sigma)$ is canonically isomorphic to a symplectic space of solutions to $Pu=0$ \cite[Proposition 3.1]{Wrochna}. Moreover, the dual $\mc{X}'$ is canonically identified with the distributional solutions $u\in H^{1,-\infty}_{0,b,loc}(M)$ to $Pu=0$.

\subsection{Bulk-to-boundary correspondence}

Let $\nu_{\pm}:=\frac{d}{2}\pm\nu$ be the indicial roots of $P$. Throughout we assume the Breitenlohner-Freedman bound $\nu>0$ \cite{BZ82} and we denote $\nu_+$ by $\Delta$. By \cite[Proposition 3.5]{DW} (see also \cite[Proposition 3.7]{Wrochna}), any distributional solution $u\in\mc{X}'$ admits a representation of the form $u=z^{\Delta}v$ for some $v\in C^\infty([0,\ep)_z;\mc{D}'(\partial M))$, and the  \textit{bulk-to-boundary} map 
\begin{align}
    \partial:\mc{X}'\ni u\mapsto v|_{\partial M}\in \mc{D}'(\partial M)
\end{align}
is continuous (with respect to the natural topologies on $H^{1,-\infty}_{0,b,loc}(M)$ and $ \mc{D}'(\partial M)$). 

We now show that the bulk-to-boundary map intertwines the canonical actions of AdS isometries in the bulk and their conformal manifestations on the boundary. To this end, let
\begin{equation}\label{e:scalar}(x,y)=x^0y^0+x^{d+1}y^{d+1} - x^1y^1 - \cdots - x^{d}y^{d}=x^0y^0+x^{d+1}y^{d+1}-\vec{x}\cdot\vec{y},\end{equation}
where $\vec{x}=(x^1,...,x^{d})$. Let $G=SO(2,d)$ denote the group of real linear transformations of $\R^{d+2}$ which preserve the scalar product, given in equation \eqref{e:scalar}. Let $G_0$ denote the identity component and  $\widetilde{G_0}$ denote its universal cover. Then $\widetilde{G_0}$ acts as isometries on $M=(AdS_{d+1},g)$, and conformal diffeomorphisms on $(AdS_{d+1},\tilde{g})$. Note that the conformal factor $\Omega_{\Lambda}$ for $\Lambda\in \widetilde{G_0}$ satisfies
\begin{equation}\label{e:conf}\Omega_{\Lambda}^2\tilde{g}=\tilde{g}\circ\Lambda=(z\circ\Lambda)^2g=\bigg(\frac{z\circ\Lambda}{z}\bigg)^2\tilde{g}, \ \ \ \textnormal{on $AdS_{d+1}^o$}.
\end{equation}

Since each $\Lambda\in \widetilde{G_0}$ leaves the boundary $\partial M$ invariant, and the canonical action $\widetilde{G_0}\curvearrowright L^2(M,g)$,
\begin{equation}\label{e:action}
    \Lambda \cdot u=u\circ\Lambda^{-1}, \qquad 
    \Lambda\in \widetilde{G_0},
\end{equation}
is isometric, it follows that equation \eqref{e:action} induces  continuous actions on the Sobolev spaces $H^{k}_0(M)$ and $H^{k,s}_{0,b}(M)$. One then obtains continuous actions on $H^{k,\infty}_{0,b,c}(M)$ and $H^{k,\infty}_{0,b,loc}(M)$ in the canonical manner and by duality, continuous actions on $H^{-k,-\infty}_{0,b,loc}(M)$ and $H^{-k,-\infty}_{0,b,c}(M)$. Since each $\Lambda\in\widetilde{G_0}$ commutes with the Klein--Gordon operator $P$ on the pertinent function space, it follows by uniqueness of $P^{-1}_{\pm}$ that 
\begin{align}
    \Lambda^{-1}P^{-1}_{\pm}\Lambda = P^{-1}_{\pm} .
\end{align}
In particular, $\Lambda^{-1}G\Lambda = G$, so that $\Lambda$ induces via equation \eqref{e:action} a symplectic automorphism of $\mc{X}$. 

\begin{lem}\label{l:intertwine} For each $\Lambda\in\widetilde{G_0}$, $u\in\mc{X}'$,
\begin{equation}
\label{e:intertwine}
   \partial(\Lambda'(u))= \Omega^{\Delta}|_{\partial M}\cdot(\Lambda|_{\partial M})'\partial(u),
\end{equation}
where $\Lambda'$ is the dual action on $\mc{X}'$, and $\cdot$ denotes multiplication of distributions by smooth functions.
\end{lem}

\begin{proof} 
Throughout the proof we let $W=(0,\ep)_z\times\partial M_x$ and $N=\Lambda^{-1}(W)$. Write 
$$\Lambda_1=\pi_1\circ\Lambda ,$$ 
where $\pi_1:W\to(0,\ep)$ is the projection onto the first coordinate. Then, $\Lambda_1$ is a submersion from $N\to(0,\ep)$. Noting that the dual action of $\Lambda$ on distributions is the pushforward by $\Lambda^{-1}$, it follows by properness of $\Lambda^{-1}$ and \cite[Proposition 2.18]{LMV} that for any $v\in C^\infty((0,\ep)_z;\mc{D}'(\partial M))$, $\Lambda'(v)$ is transversal on $N$ with respect to $\Lambda_1$ \cite{AS} (see also \cite[Definition 2.3]{LMV}), meaning $(\Lambda_1)_*(\Lambda'(v).F)\in C_c^\infty((0,\ep))$ for all $F\in C_c^\infty(N)$. It follows that \begin{equation}\label{e:transversal}
    (\Lambda_1)_*(\Lambda'(v).F)(z)=\la v(z),(\Lambda\cdot F)(z,\cdot)\ra = \la(\Lambda^{-1}_z)_*(v(z)),F|_{\Lambda_1^{-1}\{z\}\cap N} \ra, \ \ \ z\in (0,\ep),
\end{equation}
where $\Lambda^{-1}_z:\partial M\to \Lambda_1^{-1}\{z\}\cap N$ is simply $\Lambda^{-1}_z(x)=\Lambda^{-1}(z,x)$. 
Indeed, given $g\in C^\infty_c((0,\ep))$ and $(z,x)\in W$, we have
$$(F\cdot(g\circ\Lambda_1))\circ\Lambda^{-1}(z,x)=F(\Lambda^{-1}(z,x))g(z),$$
implying
\begin{align*}\la (\Lambda_1)_*(\Lambda'(v).F),g\ra&=\la v,(F\cdot(g\circ\Lambda_1))\circ\Lambda^{-1}\ra\\
&=\la v,(\Lambda\cdot F)(g\ten 1)\ra\\
&=\int \la v(z),(\Lambda\cdot F)(z,\cdot)\ra g(z) \ dz.
\end{align*}
Smoothness entails equation \eqref{e:transversal}. By \cite[Proposition 2.7]{LMV}, it follows that $$\Lambda'(v) = ((\Lambda^{-1}_{z})_*(v(z))_{z\in(0,\ep)}$$
is a smooth family of distributions on the fibers $\Lambda_1^{-1}\{z\}\cap N$ (in the sense of \cite[Definition 2.6]{LMV}).

Let $u\in\mc{X}'$ be of the form $u=z^{\Delta}v$,  $v\in C^\infty([0,\ep)_z;\mc{D}'(\partial M))\subseteq C^\infty((0,\ep)_z;\mc{D}'(\partial M))$ and fix $f\in C^\infty_c(\partial M)$. Define $F:M\to\R$ by $F(z,x)=f(x)$. On the one hand, the family $(\Lambda\cdot F(z,\cdot))_{z\in(0,\ep)}$ is bounded in the barrelled  locally convex space $C_c^\infty(\partial M)$ (see \cite[Definition 33.1]{Treves}). Moreover, 
$$\lim_{z\to0}\Lambda\cdot F(z,x)=\Lambda|_{\partial M}\cdot f(x),$$
uniformly in $\mathrm{supp}(f)$. The same is true for all derivatives of $f$ (and corresponding partials of $F$), so that 
$$\lim_{z\to0}\Lambda\cdot F(z,\cdot)=\Lambda|_{\partial M}\cdot f  \ \ \ \textnormal{in $C_c^\infty(\partial M)$}.$$
Since $v$ is continuous and the spaces $C^\infty_c(\partial M)$ and $\mc{D}'(\partial M)$ are barrelled, it follows that the dual pairing
$$\la\cdot,\cdot\ra:\mc{D}'(\partial M)\times C_c^\infty(\partial M)\to\R$$
is hypocontinuous \cite[Theorem 41.2]{Treves}, implying, by boundedness of the family $(\Lambda\cdot F(z,\cdot))_{z\in(0,\ep)}$, that
$$\la v(0),\Lambda|_{\partial M}\cdot f\ra=\lim_{z\to 0}\la v(z),\Lambda\cdot F(z,\cdot)\ra=\lim_{z\to 0}\la (\Lambda^{-1}_z)_*(v(z)),F|_{\Lambda_1^{-1}\{z\}\cap N}\ra .$$

On the other hand, $\Lambda'(u)\in\mc{X}'$, so is of the form $z^{\Delta}v_{\Lambda}$, with $v_{\Lambda}\in C^\infty([0,\ep)_z;\mc{D}'(\partial M))$. But
$$z^{\Delta}v_{\Lambda}=\Lambda'(u)=\Lambda'(z^{\Delta}v)=(z\circ\Lambda)^{\Delta}\Lambda'(v), \ \ \ z\in(0,\ep),$$
so that 
$$v_\Lambda=\Omega^{\Delta}\cdot\Lambda'(v) = (\Omega^{\Delta}_z\cdot(\Lambda^{-1}_{z})_*(v(z))_{z\in(0,\ep)} ,$$ 
where $\Omega^{\Delta}_z=\Omega^{\Delta}|_{\Lambda_1^{-1}\{z\}\cap N}$. By smoothness of $\Omega^{\Delta}_z$, $\Lambda^{-1}_{z}$ and $v$ in $z$, it follows that
$$\la \partial(\Lambda'(u)),f\ra=\la v_{\Lambda}(0),f\ra=\lim_{z\to 0}\la \Omega^{\Delta}_z\cdot(\Lambda^{-1}_{z})_*(v(z)),F|_{\Lambda_1^{-1}\{z\}\cap N}\ra .$$
Putting things together, we see that
\begin{align*}\la(\Lambda|_{\partial M})'\partial (u),f\ra&=\la\partial (u),\Lambda|_{\partial M}\cdot f\ra=\la v(0),\Lambda|_{\partial M}\cdot f\ra\\
&=\lim_{z\to 0}\la (\Lambda^{-1}_z)_*(v(z)),F|_{\Lambda_1^{-1}\{z\}\cap N}\ra\\
&=\lim_{z\to 0}\la \Omega^{-\Delta}_z\Omega^{\Delta}_z\cdot(\Lambda^{-1}_{z})_*(v(z)),F|_{\Lambda_1^{-1}\{z\}\cap N}\ra\\
&=\la \Omega^{-\Delta}|_{\partial M}\cdot \partial(\Lambda'(u)),f\ra.
\end{align*}

\end{proof}

\subsection{Subregion duality and quantum error correction}
\label{s:subregionQEC}

Let $\eta:\mc{L}(\mc{X},\mc{X}')$ be a continuous, positive,  symmetric linear map satisfying
\begin{equation}\label{e:sigma}|\sigma(u_1,u_2)|\leq\la u_1,\eta(u_1)\ra^{1/2}\la u_2,\eta(u_2)\ra^{1/2}, \ \ \ u_1,u_2\in \mc{X}.\end{equation}
By definition of $\mc{X}$ as a quotient space, $\eta$ must be a bi-solution:
\begin{align}
    P\circ\eta=0\quad\text{on}\quad H^{-1,\infty}_{0,b,c}(M),\qquad
    \eta\circ P=0\quad\text{on}\quad H^{1,\infty}_{0,b,c}(M) .
\end{align} 
Moreover, continuity means that
\begin{align}
    \eta:H^{-1,\infty}_{0,b,c}(M)\to H^{1,-\infty}_{0,b,loc}(M)
\end{align}
is continuous.
The above properties are satisfied by the covariances of any holographic Hadamard state in the sense of \cite{Wrochna}. In particular, for the covariance of the global vacuum on $M$, which is a ground state for the static dynamics (see \cite[\S4.2]{Wrochna} and \cite{G19,KayWald,DG} for details).

We let $\mc{X}^{cpl}$ denote the completion of $\mc{X}$ with respect to the Euclidean norm induced by $\eta$. The boundedness conditions in equation \eqref{e:sigma} ensure that $\sigma$ extends to a (pre-)symplectic form $\sigma^{cpl}$ on $\mc{X}^{cpl}$. By \cite[Lemma 2.1]{DW}, it follows that the restriction of $\partial':\mc{D}''(\partial M)\to\mc{X}''$ to $C_c^\infty(\partial M)$ satisfies $\partial'(C^\infty_c(\partial M))\subseteq\mc{X}^{cpl}$. 

For open subsets $O\subseteq\partial M$ and $V\subseteq M$, we define 
\begin{align}
\begin{aligned}
    \mc{X}_{bd}(O) & :=C_c^\infty(O) ,\\
    \mc{X}_{bulk}(V) & :=H^{-1,\infty}_{0,b,c}(V;\R)/(PH^{1,\infty}_{0,b,c}(M;\R)\cap H^{-1,\infty}_{0,b,c}(V;\R)),
\end{aligned}
\end{align}
which becomes a (pre-)symplectic space when canonically viewed as a subspace of $(\mc{X},\sigma)$. Let
\begin{align}
    \mc{A}_{bd}(O):=CCR(\overline{\partial'(C_c^\infty(O))},\sigma^{cpl}) \quad\textnormal{and}\quad \mc{A}_{bulk}(V):=CCR(\overline{\mc{X}_{bulk}(V)},\sigma^{cpl})
\end{align}
be the Weyl $C^*$-algebras generated by the (pre-)symplectic spaces $(\overline{\partial'(C_c^\infty(O))},\sigma^{cpl})$ and $(\overline{\mc{X}_{bulk}(V)},\sigma^{cpl})$, respectively, where the closures are with respect to $\eta$. Then, $\{\mc{A}_{bd}(O)\}$ and $\{\mc{A}_{bulk}(V)\}$ define isotonic nets of $C^*$-algebras satisfying
$\mc{A}_{bulk}(V(O))\subseteq\mc{A}_{bd}(O)$ by \cite[Theorem 3.7]{DW}, where $V(O)$ is the \textit{domain of uniqueness} of $O\subseteq\partial M$, defined as the maximal open subset of $M$ such that for any  $u\in H^{1,-\infty}_{0,b,loc}(M)$ satisfying $Pu=0$ on $M$, we have $\partial u=0$ on $O$ implies $u=0$ on $V(O)$. 
As mentioned in \cite[\S 3.5]{DW}, since (the conformal image of) $(M,g)$ is analytic, it follows from Holmgren's theorem \cite[Theorem 8.6.5]{Horm} that $V(O)\neq\emptyset$ for arbitrary open $O\subseteq\partial M$. 

Following Ribeiro's notation \cite[\S1.2.2]{Ribeiro07}, to which we refer the reader for details, given $r\in \partial M$, let $\overline{r}$ denote its antipodal point, the unique point where all null geodesics emanating from $r$ will refocus. Let 
\begin{align}
    \mathrm{Min}(r):=I^+(r,\partial M)\cap I^{-}(\overline{\overline{r}},\partial M)
\end{align}
be the associated Minkowski domain to the future of $r\in\partial M$ (see \cite[Figure 1.2]{Ribeiro07}). Given $p,q\in \mathrm{Min}(r)$ with $q\in I^+(p,\partial M)$, let
\begin{align}
    \mc{D}_{p,q}:=I^+(p,\partial M)\cap I^{-}(q,\partial M),\quad
    \mc{W}_{p,q}:=I^+(p,M)\cap I^{-}(q,M)\cap M^o
\end{align}
respectively denote the associated boundary diamond and the bulk causal wedge. 



\begin{prop}\label{p:inclusion} $\mc{W}_{p,q}\subseteq V(\mc{D}_{p,q})$. Consequently, $\mc{A}_{bulk}(\mc{W}_{p,q})\subseteq\mc{A}_{bd}(\mc{D}_{p,q})$. 
\end{prop}

\begin{proof} Suppose $u\in\mc{X}'$ satisfies $\partial(u)=0$ on $\mc{D}_{p,q}$. Let $(\Lambda_{p,q}(t))_{t\in\R}$ denote the canonical one-parameter group of isometries preserving the wedge $\mc{W}_{p,q}$ \cite[Equation (1.26)]{Ribeiro07} (with notational change $u=\Lambda$, $\lambda=t$). Since $\Lambda_{p,q}(t)|_{\partial M}$ leaves the  boundary diamond $\mc{D}_{p,q}$ invariant, by Lemma \ref{l:intertwine}, we have
$$\partial(\Lambda_{p,q}(t)'(u))=\Omega_{p,q}(t)^{\Delta}\cdot(\Lambda_{p,q}(t)|_{\partial M})'\partial(u)=0$$
on $\mc{D}_{p,q}$ for all $t\in\R$, where $\Omega_{p,q}(t)$ is the conformal factor of $\Lambda_{p,q}(t)$. Hence, $\Lambda_{p,q}(t)'(u)=0$ on $V(\mc{D}_{p,q})$ for all $t\in\R$. In other words, $u=0$ on $\Lambda_{p,q}(t)(V(\mc{D}_{p,q}))$ for all $t\in\R$. 

Pick $x\in\mc{W}_{p,q}$. Then, $\Lambda_{p,q}(t)(x)\to\mc{D}_{p,q}$ as $t\to\infty$. Since $V(\mc{D}_{p,q})$ is an open neighborhood of $\mc{D}_{p,q}$, we must have $U_x\subseteq\Lambda_{p,q}(t)(V(\mc{D}_{p,q}))$ for some $t$ and some neighbourhood $U_x$ of $x$. By above, $u=0$ on $U_x$, and since $x\in\mc{W}_{p,q}$ was arbitrary, $u=0$ on $\mc{W}_{p,q}$.
\end{proof}

For the rest of this section we let $\eta$ be the covariance of the global AdS vacuum state $\omega\in\A_{bulk}(M)^*$, which is a ground state for the Klein--Gordon dynamics (see, e.g., \cite[\S4.10.3]{G19}, \cite[\S18.3.2]{DG} or \cite[Lemma 4.4]{Wrochna} for related formulae in the complex formalism). Both $\eta$ and $\omega$ are invariant under the pertinent $\widetilde{G}_0$ action. Let $(\mc{H},\pi_F,\Omega)$ denote the associated GNS construction, which is a Fock representation (see, e.g., \cite[\S 4.9]{G19}). By invariance, the $\widetilde{G}_0$-action is unitarily implemented via $U:\widetilde{G}_0\to\mc{U}(\cH)$.


\begin{thm}\label{t:KG} Under the above hypothesis and notation, there exists an isometry 
$$V:L^2(\mc{A}_{bulk}(\mc{W}_{p,q}),\om)\to L^2(\mc{A}_{bd}(\mc{D}_{p,q}),\om)$$
between the GNS Hilbert spaces of $\om|_{\mc{A}_{bulk}(\mc{W}_{p,q})}$ and $\om|_{\mc{A}_{bd}(\mc{D}_{p,q})}$, such that the equivalent conditions of Theorem \ref{t:main} are satisfied with respect to the von Neumann algebras 
$$\mc{A}:=\pi_{\om}(\mc{A}_{bd}(\mc{D}_{p,q}))''\quad\text{and}\quad\mc{B}:=\pi_{\om}(\mc{A}_{bulk}(\mc{W}_{p,q}))'' .$$ 
\end{thm}

The proof amounts to showing that the bulk and boundary 2-point functions arising from the framework of \cite{DW} coincide with the known bulk and boundary 2-point functions for the global AdS vacuum, at which point one obtains KMS-conditions for the pertinent $C^*$-algebras, allowing one to appeal to a recent result of \cite[Theorem 1.1]{Gesteau:2020rtg} to conclude.

\begin{proof} 

In the GNS/Fock space the global 2-point function satisfies 
$$W_2(u,v):=\la\Omega,\vphi_F(u)\vphi_F(v)\Omega\ra=\la u,\eta^{cpl}(v)\ra+\frac{i}{2}\la u,\sigma^{cpl}(v)\ra, \ \ \ u,v\in\mc{X}^{cpl},$$
where $\om=\om_{\Omega}\circ\pi_{F}$, $\pi_F(W(x))=W_F(x)$, $\pi_F(\vphi(u))=\vphi_F(u)$ are the usual Fock space Weyl/field operators \cite[\S 4.9.3]{G19}. By the standard procedure for obtaining vacuum propagators in static spacetimes (see, e.g., \cite{DGass,DG,IW03}), it follows that in the interior $M^o$, the left hand side above coincides with the known 2-point function of the global AdS vacuum, which, for $x,x'$ in the fundamental sheet of $M^o$, satisfies
\begin{equation}\label{e:G}G^+(x,x')=\frac{\sqrt{\pi}\Gamma(d/2+\nu)}{\sqrt{2}(2\pi)^{(d+1)/2}2^\nu}\mc{Z}_{\frac{d+1}{2}-1,\nu}(x\cdot x'+i0),\end{equation}
where $\mc{Z}_{\frac{d+1}{2}-1,\nu}$ is the Gegenbauer function (see \cite[\S7]{DGass}, and note the opposite spacetime signature convention and $d+1$ dimensions). As the notation suggests, we view (\ref{e:G}) as the boundary limit, in the sense of distributions, of analytic functions in a suitable complexified tube domain (see below). Using the known relationship with Legendre functions of the second kind $\mc{Q}^\mu_\nu$ (\cite[Equation (B5)]{DGR}), and the fact that we are in the fundamental sheet, one easily sees that the above 2-point function coincides with (the real boundary limit of) the analytic 2-point Klein-Gordon functions from \cite[\S 6.1]{BEM} (and \cite[\S 4.2]{BBMS}), namely
$$G^+(z,z')=\frac{e^{-i\pi\frac{d-1}{2}}}{(2\pi)^{(d+1)/2}}(\zeta^2-1)^{-\frac{d-1}{4}}\mc{Q}^{\frac{d-1}{2}}_{\nu-1/2}(\zeta),$$
where $\zeta=z\cdot z'$ when $z\in T^{-}$ and $z'\in T^{+}$, where
\begin{align*}
T^{-}&=\{z=x+iy\in AdS_{d+1}^{c}\mid y^2>0, \epsilon(z)=1\},\\
T^{+}&=\{z=x+iy\in AdS_{d+1}^{c}\mid y^2>0, \epsilon(z)=-1\},
\end{align*}
$AdS_{d+1}^{c}$ is the complexified AdS, and 
\begin{align*}
    \epsilon(z)=\mathrm{sign}(y^0x^{d+1}-x^0y^{d+1}) .
\end{align*}
Here, $T^{\pm}$ are the AdS versions of the usual forward and backward tubes in Minkowski space (see \cite[\S4.1]{BBMS} and/or \cite[\S3]{BEM} for an extended discussion). Hence, the (quasi-free) AdS ground state 2-point function and associated $n$-point functions satisfy the tempered spectral condition of \cite[\S 4]{BEM}, as well as the dimensional boundary condition at infinity of \cite[\S 2.2]{BBMS}. In particular, by \cite{BBMS}, the boundary limit
$$W^\infty_2(f,g):=\lim_{z_1,z_2\to0}z_1^{-\Delta}z_2^{-\Delta}\widetilde{W}_2(\delta_{z_1}\ten f,\delta_{z_2}\ten g)$$
defines a causal, conformally covariant 2-point function on the boundary $\partial M$, where $\widetilde{W}_2$ is the natural extension of $W_2$ defined by the distributional limit of analytic functions given above (see also \cite{DR03} for similar results). Next, we show that $W^\infty_2$ coincides with $\la\Omega,\vphi_F(\partial'
(\cdot))\vphi_F(\partial'(\cdot))\Omega\ra$, the boundary 2-point function arising from \cite{DW}. 

By \cite[Lemma 2.1]{DW} $\eta:\mc{X}\to\mc{X}'$ extends continuously to $\eta^{cpl}:\mc{X}^{cpl}\to\mc{X}'$. For $z\in(0,\ep)$ and $f\in C^\infty_c(\partial M)$, the functional $z^{-\Delta}(\delta_z\ten f):\mc{X}'\to\R$ is continuous by virtue of the inclusion $\mc{X}'\subseteq z^{\delta}C^\infty([0,\ep);\mc{D}'(\partial M))$ (see \cite[Proposition 3.5]{DW}), and it follows that 
$$\partial'(f)=\lim_{z\to0}z^{-\Delta}\delta_z\ten f$$
weakly in $\mc{X}^{cpl}$. Indeed, for $u\in\mc{X}^{cpl}$, $\eta^{cpl}(u)\in\mc{X}'$ and we have
\begin{align*}\la\partial'(f),\eta^{cpl}(u)\ra&=\la f,\partial(\eta^{cpl}(u))\ra=\lim_{z\to 0}z^{-\Delta}\la f,\eta^{cpl}(u)(z)\ra\\
&=\lim_{z\to 0}z^{-\Delta}\la \delta_z\ten f,\eta^{cpl}(u)\ra.
\end{align*}
Thus, for $f,g\in C^\infty_c(\partial M)$, we have
$$\la\partial'(f),\eta^{cpl}(\partial'(g))\ra=\lim_{z_1,z_2\to0}z_1^{-\Delta}z_2^{-\Delta}\la\delta_{z_1}\ten f,\eta^{cpl}(\delta_{z_2}\ten g)\ra$$
Since $\sigma^{cpl}=-2\eta^{cpl}j^{cpl}$ \cite[Proposition 4.9.1]{G19} for an anti-involution $j^{cpl}\in O(X^{cpl},\eta^{cpl})$, a similar limiting expression applies to $\la\partial'(f),\sigma^{cpl}(\partial'(g))\ra$, and therefore to the 2-point function
$$\la\Omega,\vphi_F(\partial'(f))\vphi_F(\partial'(g))\Omega\ra=\lim_{z_1,z_2\to0}z_1^{-\Delta}z_2^{-\Delta}\widetilde{W}_2(\delta_{z_1}\ten f,\delta_{z_2}\ten g)=W^\infty_2(f,g).$$
By causality and conformal invariance of $W^\infty_2$, together with essential self-adjointness of the field operators $\vphi_F(\partial'(f))$ (over the standard finitely supported Fock domain) it follows that the associated net $\{\pi_F(\mc{A}_{bd}(O))''\}_{O\subseteq\mathrm{Min}(r)}$ defines a causal, conformally covariant pre-cosheaf of von Neumann algebras on the Minkowski patch $\mathrm{Min}(r)$. By \cite[Theorem 2.3(ii)]{BGL}, for any diamond region $\mc{D}_{p,q}\subseteq \mathrm{Min}(r)$ the vacuum modular flow $(\pi_F(\mc{A}_{bd}(\mc{D}_{p,q}))'',\om_{\Om})$ is geometrically implemented through $(U(\Lambda_{p,q}(t)))_{t\in\R}$, where $(\Lambda_{p,q}(t))_{t\in\R}$ is the one-parameter subgroup from \cite[Equation (1.26)]{Ribeiro07}. The vacuum $\om_\Omega$ is then a KMS state for the automorphism group $(\sigma^{p,q}_t)_{t\in\R}$ on $\pi_F(\mc{A}_{bd}(\mc{D}_{p,q}))''$. By unitary implementation of the action, it follows that $\sigma^{p,q}_t$ leaves the subalgebra $\pi_F(\mc{A}_{bulk}(\mc{W}_{p,q}))$ invariant. We are therefore in position to apply \cite[Theorem 1.1]{Gesteau:2020rtg} to achieve complementary recovery of the inclusion $\pi_F(\mc{A}_{bulk}(\mc{W}_{p,q}))\subseteq\pi_F(\mc{A}_{bd}(\mc{D}_{p,q}))$.
\end{proof}

\section{Outlook}\label{s:outlook}

Several natural lines of investigation are suggested by our work, which we now summarize.
\begin{itemize}
\item The framework of \cite{DW} utilized in Section \ref{s:4} applies to a large class of asymptotic anti-de Sitter spaces (in the sense of \cite[Definition 3.1]{DW}, see also \cite{Vasy,Wrochna}). It is of course natural to extend the results of Section \ref{s:4} to that setting. Techniques surrounding the time-like tube theorems of \cite{Stroh,SW24} could likely be relevant, together with the geometric foliations studied in \cite{Ribeiro07}. We plan to investigate this in future work.
\item The properties of (the dual of) the bulk-to-boundary map $\partial$, together with the established inclusion of bulk wedge algebras into boundary diamond algebras (Proposition \ref{p:inclusion}) suggests a potential connection with the notion of  projective holography studied in \cite{DR03} (see \cite{Rehren:1999jn} for the notion of algebraic holography).
\item Additional examples of operator algebraic subregion-subregion duality, for instance, in boundary chiral U(1)-current conformal field theories \cite{BMT}.
\end{itemize}

\end{document}